\documentclass[11pt]{article}

\usepackage[utf8]{inputenc}
\usepackage[T1]{fontenc}

\usepackage[margin=1in,footskip=0.25in]{geometry}

\usepackage{graphicx}
\graphicspath{{Figures/}}
\usepackage{subfig}
\usepackage{adjustbox}

\usepackage{amsmath,amssymb,amsthm}
\newtheorem{prop}{Proposition}

\usepackage{booktabs}
\usepackage{multirow}
\usepackage{multicol}
\usepackage{makecell}
\usepackage{comment}
\usepackage{diagbox}
\usepackage{siunitx}
\usepackage{colortbl}
\usepackage{xcolor}
\definecolor{UnimoRed}{RGB}{209,65,36}

\usepackage[ruled,vlined,linesnumbered]{algorithm2e}

\usepackage{tikz}
\usetikzlibrary{matrix,positioning,arrows.meta,arrows,fit,backgrounds,arrows.meta,shapes,calc}
\usepackage{tikz}
\usetikzlibrary{matrix,positioning}
\usepackage[customcolors]{hf-tikz}
\tikzset{
  mymat/.style={
    matrix of math nodes,
    text height=2.5ex,
    text depth=0.75ex,
    text width=3.25ex,
    align=center,
    column sep=-\pgflinewidth
  },
  mymats/.style={
    mymat,
    nodes={draw,fill=#1}
  }
}

\usepackage{enumitem}
\usepackage{varioref}
\usepackage[authoryear]{natbib}
\usepackage[nottoc]{tocbibind}

\usepackage{hyperref}
\usepackage{xspace}
\usepackage{cancel}
\usepackage{url}
\usepackage{todonotes}

\usepackage{authblk}

\title{\Large\bf An Exact Combinatorial Branch-and-Bound Algorithm for the Job Sequencing and Tool Switching Problem}

\author[1]{Alberto Locatelli\thanks{Corresponding author: \texttt{alberto.locatelli@unimore.it}}}
\author[2]{Jean-François Côté}
\author[2]{Leandro C. Coelho}

\affil[1]{DISMI, University of Modena and Reggio Emilia, Reggio Emilia, Italy}
\affil[2]{CIRRELT, Université Laval, Québec, Canada}
\date{}

\begin{document}

\maketitle 

\begin{abstract}
The Job Sequencing and Tool Switching Problem (SSP) is a well-known combinatorial optimization problem arising in the context of flexible manufacturing. Since the seminal work of Tang and Denardo (1988), the SSP has received significant attention in the literature, leading to the development of numerous exact and heuristic approaches. Despite these efforts, several benchmark instances proposed decades ago and containing only 20 jobs have remained unsolved to proven optimality.
In this work, we propose an exact algorithm for the SSP, namely the Combinatorial Branch-and-Bound (C-B\&B) algorithm, which combines two distinct branch-and-bound algorithms, each introducing novel features compared with the existing literature. The former relies on a new branching scheme designed to reduce the size of the implicit enumeration tree, together with a collection of new bounding functions. The latter builds on the branching scheme introduced by Laporte et al. (2004) and strengthens it with a new bounding function and two dominance rules. Within C-B\&B, these exact algorithms are complemented by a preprocessing phase that incorporates a new branch-and-bound-based heuristic capable of rapidly generating a high-quality initial incumbent solution.
Extensive computational experiments show that C-B\&B represents a strong breakthrough over previously published approaches, proving optimality for more instances with significantly less computational effort and closing several benchmark instances that have remained open for decades.

\end{abstract}

\section{Introduction}\label{Sec_Into}
In the Job Sequencing and Tool Switching Problem (SSP), we are given a set $T=\{1,\dots,m\}$ of tools and a set $J=\{1,\dots,n\}$ of jobs that must be processed sequentially on a single flexible machine equipped with a tool magazine capable of holding at most $c$ tools at a time. Each job $j\in J$ requires a specific toolset $T_j\subseteq T$, all of which must be loaded into the magazine before the job can be processed. Since the total number of tools required across all jobs exceeds the magazine capacity (i.e., $c<m$), it may be necessary to perform some tool switches between consecutive jobs, each consisting of removing one tool from the magazine and inserting another one in its place. The SSP consists of determining both the job processing sequence and, for each job, the set of tools to be loaded into the magazine during its processing, with the objective of minimizing the total number of tool switches. Accordingly, the SSP can be decomposed into two interconnected sub-problems: the Job Sequencing Problem (JSP), which aims to find an optimal job sequence, and the Tool Replacement Problem (TRP), which consists of determining, for each sequenced job, an optimal subset of tools to be loaded into the magazine during its processing. While the JSP is NP-hard (see, e.g., \citealt{CKOS1994}), the TRP can be solved in polynomial time using, for instance, the well-known Keep Tool Needed Soonest (KTNS) policy introduced by \citet{TD1988} or the more recent and even faster IGAFull algorithm proposed by \citet{Qiu2026}.

Although, to the best of our knowledge, the earliest application related to the SSP dates back to an article by \citet{Belady1966} on computer memory management, the problem was formally stated only by \citet{TD1988} in the context of efficient tool management in Flexible Manufacturing Systems (FMS). In this setting, a computer numerical control machine (hereafter referred to simply as a machine) is used to manufacture a set of products (hereafter referred to as jobs), each requiring a specific set of tools (e.g., cutting blades, milling cutters, drill bits, etc.). The machine is equipped with a capacitated tool magazine capable of holding only a limited number of tools simultaneously. If, before processing a job, some of its required tools are not loaded in the magazine, a tool-changing device performs a sequence of tool switches to load them. Each tool switch removes a tool that the job does not require and inserts in its place a required one, retrieved from the storage area, where all tools not currently loaded in the magazine are stored. This device enables the machine to remain flexible by dynamically adjusting the configuration of its tool magazine in order to process jobs with different tool requirements. However, this operational flexibility often comes at a cost: in many practical settings, tool switches are time-consuming and can significantly affect total production time (see, e.g., \citealt{L2023}). In this context, the SSP naturally arises, since minimizing the total number of tool switches is crucial for improving FMS efficiency and productivity.

Over the last three decades, the SSP has attracted considerable attention in the literature. As evidence of this, the comprehensive survey by \citet{C2019} includes more than 60 bibliographic references, while several additional relevant contributions have appeared since its publication (see Section~\ref{sec:LiteratureReview}). Beyond its theoretical interest, this attention is largely motivated by the wide range of real-world applications the SSP has found over the years across several industrial sectors, including printed circuit board manufacturing (see, e.g., \citealt{Privault1995}), computer memory management (see, e.g., \citealt{Ghiani2010}), pharmaceutical packaging (see, e.g., \citealt{mutze2014}), and the printing industry (see, e.g., \citealt{ILLS2022,ILL2022}), among others. This interest has led to the development of a rich variety of exact and heuristic approaches. Despite these efforts, several benchmark instances with only 20 jobs, introduced decades ago, have yet to be solved to proven optimality.

In this work, we fill this gap by proposing a new exact algorithmic framework, called the Combinatorial Branch-and-Bound (C-B\&B) algorithm, which combines two distinct and complementary Branch-and-Bound (B\&B) algorithms. Both are purely combinatorial: branching is performed directly on partial job sequences, while bounding relies on ad hoc lower-bounding functions, without resorting to a Mixed-Integer Linear Programming (MILP) formulation. Although key to both algorithms is their combination of strong pruning capabilities and the low computational effort required to evaluate each node of their search tree, they differ in how they construct the partial job sequences. The first, the Positional-Insertion B\&B (PI-B\&B) algorithm, explores the solution space by inserting an unsequenced job into each possible position of the current partial sequence, whereas the second, the Positional-Extension B\&B (PE-B\&B) algorithm, constructs complete sequences by progressively extending a partial sequence with one unsequenced job. This structural difference gives the two algorithms complementary strengths: PI-B\&B is particularly effective when optimal solutions require many tool switches, whereas PE-B\&B is better suited to instances with relatively few tool switches and long subsequences of jobs that can be processed without tool switches. C-B\&B exploits this complementarity through a solver-selection strategy that chooses which B\&B algorithm to use according to the characteristics of the instance. These characteristics are assessed based on the incumbent solution computed during preprocessing: if its value is relatively large, PI-B\&B is selected; otherwise, PE-B\&B is used. The main contributions of this paper can be summarized as follows:
\begin{itemize}
    \item We introduce PI-B\&B, a new exact B\&B algorithm for the SSP that combines an insertion-based branching scheme designed to reduce the size of the implicit enumeration tree with a collection of new bounding functions that substantially limit the number of explored nodes.
    \item We present PE-B\&B, an exact B\&B algorithm for the SSP that builds on the branching scheme introduced by \citet{LSS2004} and strengthens it with a new bounding function and two dominance rules.
    \item We design PI-B\&B Heuristic (PI-B\&B-H), a heuristic variant of PI-B\&B that rapidly generates high-quality SSP solutions through a new completion pruning rule, a search policy prioritizing promising solutions, and an aggressive fathoming rule that may discard optimal solutions but substantially reduces the search space.
    \item We introduce a preprocessing and solver-selection strategy that identifies easy instances, generates strong incumbent solutions, and selects the final exact enumeration according to the characteristics of the instance.
    \item We propose C-B\&B, an exact algorithm for the SSP that integrates the components described above, and evaluate it on the benchmark instances from the literature. The computational results show that C-B\&B substantially improves upon the existing exact methods, being the first algorithm to solve all the benchmark instances with up to 25 jobs to proven optimality while requiring significantly less computational effort. In particular, it solves all the instances with 15 jobs in a few hundredths of a second and all the instances with 20 or 25 jobs in a few seconds, including those that had previously remained open after 1 hour of computational time.
\end{itemize}

The remainder of the paper is organized as follows. Section~\ref{sec:LiteratureReview} reviews developments in the field of SSP since 2019. Section~\ref{sec:Notation} formally defines the SSP and introduces the notation used throughout the paper. Sections~\ref{sec:B&B} and~\ref{sec:BBL} present PI-B\&B and PE-B\&B, respectively, whereas Section~\ref{sec:Heuristic} describes PI-B\&B-H. Section~\ref{sec:ExactAlgorithm} presents the overall C-B\&B algorithm. Extensive computational results on benchmark instances from the literature and comparisons with state-of-the-art algorithms are provided in Section~\ref{sec:ComputationalResults}. Finally, conclusions are drawn in Section~\ref{sec:Conclusions}.

\section{Literature Review}\label{sec:LiteratureReview}
After almost four decades of intensive research, the SSP literature has gradually expanded, with a variety of exact and heuristic solution approaches proposed over the years. For the models and solution techniques developed up to 2019, we refer the reader to the survey by \citet{C2019}.

With regard to exact algorithms proposed after 2019, \citet{daSilva2021} introduced a new integer linear programming formulation based on multicommodity flow, including new symmetry-breaking constraints. Computational results on standard benchmark instances showed that the proposed model outperforms previous formulations in terms of optimal solutions found, lower bound quality, and computing time. Subsequently, \citet{Akhundov2024} reformulated the SSP as a job grouping and sequencing problem and introduced a new multicommodity flow model enriched with symmetry-breaking cuts. To solve the model, a three-phase solution procedure was proposed. Initially, all parameters of the model are fixed except for the number of required job groups, which strongly influences the size of the model. Then, the procedure iteratively solves a restricted model, progressively refining a tight upper bound on the number of required job groups until optimality is proven. Computational experiments on classical benchmarks demonstrated that this method outperforms the approach proposed by \citet{daSilva2021}, solving up to 33\% more instances within shorter computing times. More recently, \citet{legrand2025} proposed a dynamic-programming-based algorithm that explores the solution space using the $A^*$ search algorithm, together with tighter lower bounds than those available in the existing literature. Preliminary computational experiments on standard benchmark instances showed that the proposed algorithm is competitive with the exact approaches available in the literature.

With regard to heuristic algorithms proposed after 2019, \citet{Mecler2021} introduced an efficient Hybrid Genetic Search (HGS) for the SSP. This approach uses a simple permutation-based solution representation and evaluates solution quality by jointly considering cost and contribution to population diversity during both parent and survivor selection. Moreover, a secondary objective is introduced to break ties among solutions and to promote those characterized by short intervals between two consecutive uses of the same tool, thereby favoring solutions that are more likely to yield improvements, as short intervals can be more easily filled by keeping the tool in the magazine. Computational experiments on benchmark instances show that the proposed approach significantly outperforms previous heuristics in terms of both solution quality and computational time.

Although many exact and metaheuristic methods proposed in the SSP literature rely on repeatedly solving the TRP to evaluate candidate job sequences, relatively little attention has historically been paid to reducing the computational time required by this task. For more than three decades the standard approach has been the KTNS algorithm, which solves the TRP in $O(mn)$ time. Only recently, \citet{Qiu2026} proposed a new algorithm for the TRP, namely IGAFull, an exact $O(cn)$ time algorithm that sequentially invokes two procedures, both requiring $O(cn)$ time. The former, namely IGA, computes the minimum number of tool switches required to process a given job sequence, while the latter, called ToFullMag, constructs an optimal solution by determining, for each job in the sequence, an optimal set of tools to be loaded into the magazine during processing. 
Finally, \citet{Almeida2025} reviewed the existing literature on TRP algorithms and analyzed both the classical KTNS algorithm and the IGAFull approach. The authors also proposed parallel implementations of both algorithms based on a divide-and-conquer strategy exploiting symmetry properties of TRP solutions. Computational experiments showed that parallel implementations can significantly reduce evaluation times for large instances, while for instances with fewer than 30 jobs the additional parallelization overhead makes classical serial implementations faster.

A number of SSP variants have been investigated since 2019.
\citet{Calmels2022} introduced the SSP with non-identical parallel machines (SSP-NPM), in which jobs must be assigned to and sequenced on machines characterized by different processing times, magazine capacities, and tool-switching times. To address the problem, the author proposed a MILP formulation, several construction heuristics, and an iterated local search approach. Subsequently, \citet{Cura2023} developed a hybrid algorithm combining a genetic algorithm with local search. Two variants were proposed, differing in the adopted tool-loading policy: the KTNS policy and a randomly remove tool policy. The following year, \citet{Soares2024} introduced two parallel biased random-key genetic algorithms, hybridized with tailored local search procedures organized within a variable neighborhood descent framework, to separately minimize the makespan and the total flow time. More recently, \citet{Hadj2026} focused on the exact solution of the makespan-minimization variant by improving the position-based MILP formulation of \citet{Calmels2022} and introducing new position-based and job-group-based arc-flow formulations. These models were further strengthened through lower and upper bounds, valid inequalities, and symmetry-breaking constraints. In a different direction, \citet{RIFAI2022} studied the sequence-dependent SSP, in which tool-switching times are non-uniform and depend on the specific pair of tools removed from and inserted into the same magazine slot. To address this variant, the authors proposed a two-stage heuristic that combines an adaptive large neighborhood search for the JSP with a procedure integrating the KTNS policy and simulated annealing for the corresponding TRP. More recently, \citet{Iori2024} addressed four variants of the SSP that differ in terms of tool-switching times, which may be constant or non-uniform; job sequences, which may be fixed or variable; and tool requirements, which may be ordered or unordered. In the ordered variants, the tools required by each job must be arranged in a prescribed order along the magazine slots. The authors analyzed the computational complexity of these variants and proposed dedicated arc-flow formulations for their solution.

\section{Notation and Problem Definition}\label{sec:Notation}
In this section, we introduce the notation used throughout the article. Given a sequence of jobs $\sigma = (\sigma_1, \dots, \sigma_l)$, let $J_{\sigma} = \{\sigma_1, \dots, \sigma_l\}$ be the set of jobs in $\sigma$, and let $\overline{J}_{\sigma} = J \setminus J_{\sigma}$ be the set of jobs not included in $\sigma$. The reverse sequence of $\sigma$ is denoted by $\sigma^R = (\sigma_l, \dots, \sigma_1)$. Moreover, for each $p \in \{1,\dots,l\}$, let $\sigma_{\overleftarrow{p}} = (\sigma_1, \dots, \sigma_p)$ denote the subsequence of $\sigma$ up to position $p$, while $\sigma_{\overrightarrow{p}} = (\sigma_p, \dots, \sigma_l)$ denotes the subsequence of $\sigma$ from position $p$ to the end.
Let $T_\sigma = \bigcup_{j \in J_{\sigma}} T_j$ be the set of all tools required by the jobs in $J_{\sigma}$.
Conversely, let $\overline{T}_\sigma = T \setminus T_\sigma$ denote the set of tools that are not required by any job in $J_{\sigma}$.
A set $M \subseteq T$ such that $\vert M \vert = c$ is called a magazine configuration.
A sequence of magazine configurations $\mathcal{M} = (M_1, \dots, M_l)$ is said to cover $\sigma$ if $T_{\sigma_i} \subseteq M_i$ holds ($i = 1, \dots, l$). We denote by $z(\mathcal{M}) = \sum_{i=2}^{l} \left| M_i \setminus M_{i-1} \right|$ the number of tool switches required for processing the sequence of magazine configurations $\mathcal{M}$.

Given a sequence of jobs $\sigma$, the corresponding TRP consists of determining the minimum number $z(\sigma)$ of tool switches required to process $\sigma$, together with a sequence of magazine configurations $\mathcal{M}_{\sigma}$ covering $\sigma$ such that
$z(\mathcal{M}_{\sigma}) = z(\sigma)$.
In the SSP, instead, we are given a set of jobs $J$, and the goal is to determine both a permutation $\sigma^*$ of $J$ and a sequence of magazine configurations $\mathcal{M}_{\sigma^*}$ covering $\sigma^*$ such that
$ z(\sigma^*) = z(\mathcal{M}_{\sigma^*}) =\min \{ z(\sigma) : \sigma \text{ is a permutation of } J \}$.

\section{The PI-B\&B Algorithm}\label{sec:B&B}
A B\&B algorithm is based on two main operations: branching, which involves partitioning the set of feasible solutions into subsets, and bounding, which aims to prune the enumeration by evaluating the objective value of the subproblems generated by this partition. In each of the B\&B algorithms we introduce in this paper, the nodes of the corresponding search tree are characterized by a partial sequence $\sigma$ of scheduled jobs and the set $\overline{J}_{\sigma}$ of remaining jobs that have yet to be inserted into $\sigma$. The algorithms differ in the branching scheme adopted to extend the partial sequences, and hence in the structure of the resulting search tree. For the sake of simplicity, in the remainder of the paper, we denote each node of a B\&B search tree by the corresponding sequence $\sigma$, and we call a completion of a node $\sigma$ any leaf of the subtree rooted at $\sigma$, i.e., any complete sequence of $n$ jobs that can still be generated from the partial sequence $\sigma$ in accordance with the corresponding branching scheme.

In this section, we present the bounding procedure and branching rule embedded in the proposed PI-B\&B algorithm.
Before executing the PI-B\&B algorithm, a preprocessing phase, detailed in Section~\ref{sec:ExactAlgorithm}, computes a heuristic solution $\sigma^H$ as well as the initial upper bound $ub = z(\sigma^H)$.

\subsection{Lower Bounding Functions}\label{sec:LB}
The PI-B\&B algorithm we propose relies on three new bounding techniques introduced to reduce the size of the branching tree by pruning its nodes. In this section, we present these techniques and provide a formal proof of their correctness. Before applying the bounding functions at the current node $\sigma$ of the B\&B search tree, the value $z(\sigma)$ is computed using IGA.

\subsubsection*{Bound $L_1(\sigma)$}
The first lower bound accounts for the tools that are not required by the jobs currently included in $\sigma$ but must be loaded when completing the sequence.
\begin{prop}A valid lower bound on the optimal solution of the subproblem at node $\sigma$ of the PI-B\&B tree is $
     L_1(\sigma) = z(\sigma) + \min\left\{\left| \overline{T}_\sigma \right|, m - c\right\}$.
\end{prop}
\begin{proof}
Suppose that $\vert T_\sigma\vert \leq c$, thus $z(\sigma) = 0$ and $\vert \overline{T}_\sigma \vert \geq m - c$. In this case, $L_1(\sigma)$ provides the trivial lower bound $ m - c$. On the other hand, if $\vert T_\sigma \vert > c$, then $\vert \overline{T}_\sigma \vert < m - c$ holds, and we obtain $L_1(\sigma) = z(\sigma) + \vert \overline{T}_\sigma \vert$. In this case, as $\vert T_\sigma \vert > c$, each tool in $\overline{T}_\sigma$ will require at least one additional tool switch when all jobs in $\overline{J}_{\sigma}$ are inserted into $\sigma$ and thus $z(\sigma) + \vert \overline{T}_\sigma \vert$ provides a valid lower bound on the optimal value at the node.
\end{proof}

\subsubsection*{Bound $L_2(\sigma)$}
The second lower bound is based on the following observation. In any completion of a node $\sigma$, each job in $\overline{J}_{\sigma}$ is inserted either before or after a given job $\sigma_p$ of $\sigma$. A tool required by a job inserted on one side of $\sigma_p$ may also be required by a job on the opposite side. If such a tool does not remain loaded while $\sigma_p$ is processed, it must be removed and subsequently reloaded, thereby generating an additional tool switch. Consequently, when the magazine capacity is not sufficient to keep all these tools loaded simultaneously, additional tool switches must be incurred. This yields the following result.
\begin{prop}
Let $\sigma=(\sigma_1,\dots,\sigma_l)$ be a node of the PI-B\&B tree. For each $p\in\{1,\dots,l\}$ and each subset $\Omega\subseteq\overline{J}_{\sigma}$, a valid lower bound on the optimal value of the subproblem associated with node $\sigma$ is given by
\begin{equation}\label{LB_2}
L_2(\sigma,p,\Omega)
=
z(\sigma)
+
\max\left\{
0,\,
|T_{\sigma_p}|
+\min_{S\subseteq\Omega}
\left\{
|\overrightarrow{\Gamma}_{\sigma,p,S}|
+
|\overleftarrow{\Gamma}_{\sigma,p,\Omega\setminus S}|
+
|\Gamma_{\sigma,\Omega,S}|
\right\}
+\gamma_{\sigma,p}-c
\right\},
\end{equation}
where
\begin{align}
\overrightarrow{\Gamma}_{\sigma, p, S}
&=
\left[
\left(\bigcup_{j\in S}T_j\right)
\cap T_{\sigma_{\overrightarrow{p}}}
\right]
\setminus T_{\sigma_{\overleftarrow{p}}},
\label{LB2_eq_1}\\
\overleftarrow{\Gamma}_{\sigma, p, \Omega\setminus S}
&=
\left[
\left(\bigcup_{j\in\Omega\setminus S}T_j\right)
\cap T_{\sigma_{\overleftarrow{p}}}
\right]
\setminus T_{\sigma_{\overrightarrow{p}}},
\label{LB2_eq_2}\\
\Gamma_{\sigma, \Omega,S}
&=
\left[
\left(\bigcup_{j\in S}T_j\right)
\cap
\left(\bigcup_{j\in\Omega\setminus S}T_j\right)
\right]
\setminus T_\sigma,
\label{LB2_eq_3}\\
\gamma_{\sigma, p}
&=
\max\left\{
\left|
\left(
T_{\sigma_{\overleftarrow{p}}}
\cap T_{\sigma_{\overrightarrow{p}}}
\right)
\setminus T_{\sigma_p}
\right|
-\left(z(\sigma)-|T_\sigma|+c\right),
0
\right\}.
\label{LB2_eq_4}
\end{align}
\end{prop}
\begin{proof}
Consider a job $\sigma_p \in \sigma$ and a subset $\Omega \subseteq \overline{J}_{\sigma}$. Suppose that $\sigma$ is extended by inserting all jobs in some set $S \subseteq \Omega$ before $\sigma_p$ and all remaining jobs in $\Omega \setminus S$ after $\sigma_p$.

First, consider the tools in $\overrightarrow{\Gamma}_{\sigma,p,S}$ (see Eq.~\eqref{LB2_eq_1}). Each such tool must either remain loaded in the magazine during the processing of $\sigma_p$ or induce additional tool switches beyond those already counted in $z(\sigma)$. Indeed, $\overrightarrow{\Gamma}_{\sigma,p,S}$ consists of the tools required by at least one job in $S$, which is processed before $\sigma_p$ by assumption, and by at least one job processed after $\sigma_p$, but not by any job in $\sigma_{\overleftarrow{p}}$. Thus, once a tool in $\overrightarrow{\Gamma}_{\sigma,p,S}$ is loaded to process a job in $S$, it must remain in the magazine until it is required by a job processed after $\sigma_p$, and therefore also during the processing of $\sigma_p$; otherwise, an additional tool switch is necessary. By a symmetric argument, each tool in $\overleftarrow{\Gamma}_{\sigma,p,\Omega\setminus S}$ (see Eq.~\eqref{LB2_eq_2}) must either remain loaded in the magazine during the processing of $\sigma_p$ or induce additional tool switches beyond those already counted in $z(\sigma)$.

Now, consider the tools in $\Gamma_{\sigma,\Omega,S}$ (see Eq.~\eqref{LB2_eq_3}). This set contains the tools required by at least one job in $S$ and by at least one job in $\Omega\setminus S$, but not by any job in $\sigma$. Since the jobs in $S$ are assumed to be processed before $\sigma_p$, whereas the jobs in $\Omega\setminus S$ are assumed to be processed after $\sigma_p$, each tool in $\Gamma_{\sigma,\Omega,S}$ must either remain loaded during the processing of $\sigma_p$ or be subsequently reloaded, thereby inducing an additional tool switch.

Finally, consider the tools in $(T_{\sigma_{\overleftarrow{p}}}\cap T_{\sigma_{\overrightarrow{p}}})\setminus T_{\sigma_p}$. Since at most $z(\sigma)-|T_\sigma|+c$ of these tools can be removed and subsequently reloaded without exceeding the number of tool switches already counted in $z(\sigma)$, at least $\gamma_{\sigma, p}$ (see Eq.~\eqref{LB2_eq_4}) of them must remain loaded during the processing of $\sigma_p$; otherwise, additional tool switches are necessary.

To conclude the proof, observe that, for each $S\subseteq\Omega$, the sets $T_{\sigma_p}$, $\overrightarrow{\Gamma}_{\sigma,p,S}$, $\overleftarrow{\Gamma}_{\sigma,p,\Omega\setminus S}$, $\Gamma_{\sigma,\Omega,S}$, and $(T_{\sigma_{\overleftarrow{p}}}\cap T_{\sigma_{\overrightarrow{p}}})\setminus T_{\sigma_p}$ are mutually disjoint. Consequently, since the subset $S\subseteq\Omega$ of unsequenced jobs that will be processed before $\sigma_p$ is not known at node $\sigma$, at least
$|T_{\sigma_p}|+\min_{S\subseteq\Omega}\bigl(|\overrightarrow{\Gamma}_{\sigma,p,S}|+|\overleftarrow{\Gamma}_{\sigma,p,\Omega\setminus S}|+|\Gamma_{\sigma,\Omega,S}|\bigr)+\gamma_{\sigma,p}$
tools must remain loaded during the processing of $\sigma_p$ to avoid additional tool switches. If this number exceeds the magazine capacity $c$, then at least the excess number of tools cannot remain loaded and must therefore be reloaded later. This induces at least
\[
\max\left\{0,\,
|T_{\sigma_p}|+\min_{S\subseteq\Omega}\bigl(|\overrightarrow{\Gamma}_{\sigma,p,S}|+|\overleftarrow{\Gamma}_{\sigma,p,\Omega\setminus S}|+|\Gamma_{\sigma,\Omega,S}|\bigr)+\gamma_{\sigma,p}-c
\right\}
\]
additional tool switches beyond those already counted in $z(\sigma)$. Hence, $L_2(\sigma,p,\Omega)$ (see Eq.~\eqref{LB_2}) is a valid lower bound on the optimal value of the corresponding subproblem.
\end{proof}

As the computation of $L_2(\sigma,p,\Omega)$ requires enumerating all subsets $S\subseteq\Omega$, it is exponential in $|\Omega|$. To keep the evaluation of the bound computationally efficient, we evaluate $L_2(\sigma,p,\Omega)$ over all positions $p=1,\dots,l$ and all subsets $\Omega\subseteq\overline{J}_{\sigma}$ containing at most two unsequenced jobs. Preliminary computational experiments showed that this choice provides the best trade-off between computational effort and pruning effectiveness. Formally, the second lower bound used at node $\sigma$ is defined as $
L_2(\sigma)
=
\max\left\{
L_2(\sigma,p,\Omega):
p\in\{1,\dots,l\},\
\Omega\subseteq\overline{J}_{\sigma},\
|\Omega|\leq 2
\right\}$.

\subsubsection*{Bound $L_3(\sigma)$}
The third lower bound is based on a graph relaxation of the sequencing decisions. The key observation is that, whenever two jobs $j,j'\in J$ are processed consecutively, the value $w_{j,j'}=\max\{0,|T_j\cup T_{j'}|-c\}$ provides a lower bound on the number of tool switches required between them. These pairwise lower bounds are used as edge weights in a graph associated with the subproblem at a node $\sigma$ of the PI-B\&B tree, where each edge represents the decision to process the corresponding jobs consecutively in some completion of the partial sequence $\sigma$.

More specifically, given a subset $\Omega \subseteq \overline{J}_{\sigma}$ of unsequenced jobs, let $G_{\sigma,\Omega}=(V,E)$ be the weighted graph with vertex set $V=J_{\sigma}\cup\Omega$. The edge set $E$ is obtained from the complete graph on $V$ by removing each edge $\{j,j'\}$ such that $j,j'\in J_{\sigma}$ and $j$ and $j'$ are not consecutive in $\sigma$. These edges are removed because two jobs already sequenced in $\sigma$ but not consecutive in $\sigma$ cannot become consecutive after inserting the jobs in $\Omega$.
For any spanning tree $\mathcal{T}$ of $G_{\sigma,\Omega}$, let $E(\mathcal{T})$ denote its edge set and $w(\mathcal{T})=\sum_{\{j,j'\}\in E(\mathcal{T})} w_{j,j'}$ its total weight. Since every feasible insertion of the jobs in $\Omega$ into $\sigma$ corresponds to a path in $G_{\sigma,\Omega}$ spanning all vertices in $V$, the following bound holds.
\begin{prop}
Given a node $\sigma$ of the PI-B\&B tree and a subset $\Omega \subseteq \overline{J}_{\sigma}$,
\begin{equation}\label{LB_3}
L_3(\sigma,\Omega)
=
\min\{w(\mathcal{T}) : \mathcal{T} \text{ is a spanning tree of } G_{\sigma,\Omega}\}
\end{equation}
provides a valid lower bound on the optimal value of the subproblem associated with node $\sigma$.
\end{prop}
\begin{proof}
Let $\sigma'$ be an optimal completion of $\sigma$ restricted to the jobs in $J_{\sigma}\cup\Omega$. Then, $z(\sigma')$ is a lower bound on the optimal value of the subproblem at node $\sigma$.
The consecutive pairs of jobs in $\sigma'$ induce a spanning tree $\mathcal{T}'$ of $G_{\sigma,\Omega}$. By Eq.~\eqref{LB_3}, it follows that $L_3(\sigma,\Omega)\leq w(\mathcal{T}')$. Moreover, each edge weight of $\mathcal{T}'$ is a lower bound on the number of tool switches required between the corresponding consecutive jobs of $\sigma'$, and therefore $w(\mathcal{T}')\leq z(\sigma')$. Hence, $L_3(\sigma,\Omega)\leq z(\sigma')$, which proves the claim.
\end{proof}

\begin{algorithm}
\small
\caption{$L_3(\sigma)$ calculation}
\label{alg:MST}
\KwIn{$\sigma$, $\overline{J}_{\sigma}$}
$\Omega = \overline{J}_{\sigma}$\;\label{alg:MST_Step1}
$L_3(\sigma)= L_3(\sigma,\Omega)$\;\label{alg:MST_Step2}
\While{$\Omega\neq\emptyset$}{
   Compute a minimum spanning tree $\mathcal{T}$ of $G_{\sigma,\Omega}$\;\label{alg:MST_Step3}
   $w^* = \min \{w_{j,j'} : \{j,j'\} \in E(\mathcal{T}),\ j \in \Omega\}$\;\label{alg:MST_Step4}
   \For{$j \in \Omega$}{
       $\delta(j) = \left\vert \{\{j,j'\} \in E(\mathcal{T}) : w_{j,j'} = w^*\} \right\vert$\;\label{alg:MST_Step5}
   }
   Select the job $j^* \in \Omega$ with the highest value of $\delta(j^*)$\;\label{alg:MST_Step6}
   $\Omega = \Omega \setminus \{j^*\}$\;\label{alg:MST_Step7}
   \If{$L_3(\sigma,\Omega) > L_3(\sigma)$}{
      $L_3(\sigma) = L_3(\sigma,\Omega)$\;\label{alg:MST_Step8}
   }
}
\Return{$L_3(\sigma)$}
\end{algorithm}

As the number of subsets $\Omega\subseteq \overline{J}_{\sigma}$ grows exponentially with $|\overline{J}_{\sigma}|$, computing $L_3(\sigma,\Omega)$ for each possible subset would be computationally prohibitive. Therefore, the heuristic procedure described in Algorithm~\ref{alg:MST} is used to select a promising subset $\Omega$ and compute the third proposed bound, denoted by $L_3(\sigma)$.
The procedure starts by initializing $\Omega=\overline{J}_{\sigma}$ (step~\ref{alg:MST_Step1}) and setting $L_3(\sigma)=L_3(\sigma,\Omega)$ (step~\ref{alg:MST_Step2}). Then, one job is removed from $\Omega$ at each iteration until $\Omega$ becomes empty. Specifically, a minimum spanning tree $\mathcal{T}$ of the current graph $G_{\sigma,\Omega}$ is computed (step~\ref{alg:MST_Step3}), and the minimum edge weight $w^*$ among the edges of $\mathcal{T}$ incident to a job in $\Omega$ is identified (step~\ref{alg:MST_Step4}). For each job $j\in\Omega$, $\delta(j)$ denotes the number of such incident edges with weight $w^*$ (step~\ref{alg:MST_Step5}). The job $j^*$ with the largest value of $\delta(j^*)$ is then selected and removed from $\Omega$ (steps~\ref{alg:MST_Step6} and~\ref{alg:MST_Step7}). Whenever the new value $L_3(\sigma,\Omega)$ improves the best value found so far, $L_3(\sigma)$ is updated (step~\ref{alg:MST_Step8}). Finally, the procedure returns $L_3(\sigma)$.

The computational complexity of Algorithm~\ref{alg:MST} is polynomial in $n$. Each iteration requires computing $L_3(\sigma,\Omega)$, which can be done in $O\left(n^2\log^*(n)\right)$ time using Kruskal's algorithm (see \citealt{Kruskal1956}), where $\log^*(n)$ is defined in \citet{tarjan1975}. Since at most $n-2$ iterations are performed, the overall time complexity is $O\left(n^3\log^*(n)\right)$.

\subsection{Branching Scheme}\label{sec:branching_scheme}
At level 0 of the tree, the root node corresponds to a partial sequence of two jobs. For each node $\sigma$ at level $r= 0, \dots, n-3$ of the PI-B\&B tree, there are $r+3$ potential branches, each representing a different insertion of a selected unsequenced job $j' \in \overline{J}_{\sigma}$ into $\sigma$ (see Figure~\ref{fig:BB_schema}). Finally, each node at level $n-2$ of the PI-B\&B tree is a leaf and corresponds to a complete sequence of $n$ jobs (i.e., a feasible solution to the SSP).
Note that our implicit enumeration scheme avoids generating both nodes $\sigma$ and $\sigma^R$ during the search. This is because $z(\sigma) = z(\sigma^R)$ (see \citealt{Ghiani2010}). By exploiting this symmetry, the PI-B\&B scheme effectively reduces the search space by half.

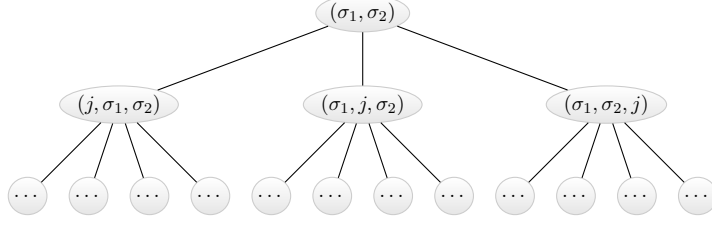
\begin{figure}
\centering
    \resizebox{0.6\textwidth}{!}{
\tikzset{
    thick,
    tree node/.style = {align=center, inner sep=0pt, font = \small},
    every label/.append style = {font=\small},
    S/.style = {draw=black!20, ellipse, minimum size = 7mm, inner sep=0pt,
                top color=white, bottom color=black!10},
    ENL/.style = {
font=\small, left=1pt},
ENR/.style = {
font=\small, right=1pt},
grow = down,  
level 1/.style = {sibling distance=4cm},
level 2/.style = {sibling distance=1cm},
level distance = 1.5cm
}
\newcommand\LB{
    \tikz\draw[very thick] (-0.5,0) -- + (1,0);}
\centering
\begin{tikzpicture}[scale=0.8, every node/.style={scale=0.7}]
\node [S] {$(\sigma_1,\sigma_2)$}
    child {node [S] {$(j,\sigma_1,\sigma_2)$}{
    child {node [S] {$\dots$}{}
        edge from parent node[ENL] {}
    }
    child {node [S] {$\dots$}
        edge from parent node[ENL] {}
    }
        child {node [S] {$\dots$}
        edge from parent node[ENL] {}
    }
    child {node [S] {$\dots$}
        edge from parent node[ENL] {}
    }
    }
        edge from parent node[ENL] {}
    }
    child {node [S] {$(\sigma_1,j,\sigma_2)$}{
    child {node [S] {$\dots$}{}
        edge from parent node[ENL] {}
    }
    child {node [S] {$\dots$}
        edge from parent node[ENL] {}
    }
        child {node [S] {$\dots$}
        edge from parent node[ENL] {}
    }
    child {node [S] {$\dots$}
        edge from parent node[ENL] {}
    }}
        edge from parent node[ENL] {}
    }
    child {node [S] {$(\sigma_1,\sigma_2,j)$}{
    child {node [S] {$\dots$}{}
        edge from parent node[ENL] {}
    }
    child {node [S] {$\dots$}
        edge from parent node[ENL] {}
    }
        child {node [S] {$\dots$}
        edge from parent node[ENL] {}
    }
    child {node [S] {$\dots$}
        edge from parent node[ENL] {}
    }}
        edge from parent node[ENL] {}
    };
\end{tikzpicture}
}
\caption{PI-B\&B tree.}\label{fig:BB_schema}
\end{figure}

\subsubsection*{Job Selection}
At the root node of the PI-B\&B tree, the initial partial sequence $\sigma = (\sigma_1, \sigma_2)$ is generated by selecting two jobs $\sigma_1, \sigma_2 \in J$ that maximize $|T_{\sigma_1} \setminus T_{\sigma_2}| + |T_{\sigma_2} \setminus T_{\sigma_1}|$, i.e., the number of tools required by exactly one of the two jobs. This selection aims to maximize the diversity between the first two jobs, potentially leading to a faster increase in the lower bounds of the nodes during the search.
At a node $\sigma = (\sigma_1, \dots, \sigma_l)$ from which we are about to branch, we select a job $j\in \overline{J}_{\sigma}$ such that $\min_{i \in \{1, \dots, l\}} \{\vert T_{\sigma_i} \cup T_{j} \vert\}$ is maximized. Observing that $\left| T_{\sigma_i} \cup T_{j} \right|$ is an approximation of the increment in the objective function due to the insertion of $j$ in $\sigma$ right before or after job $\sigma_i$ ($i=1,\dots, l$), this choice aims to select a job $j$ capable of ensuring that the minimal approximated increment in the objective function, among all child nodes, is as large as possible.

\subsubsection*{Search Strategy}
The PI-B\&B algorithm follows a depth-first search strategy. At a node $\sigma=(\sigma_1,\dots,\sigma_l)$, if the node is not fathomed by the bounding procedures (see Section~\ref{sec:LB}), a job $j\in\overline{J}_{\sigma}$ is selected according to the rule described above. Then, all $l+1$ child nodes obtained by inserting $j$ into $\sigma$ are considered sequentially, from the first insertion position to the last. Once a node is fathomed, the algorithm backtracks and continues with the next insertion position.

\section{The PE-B\&B Algorithm}\label{sec:BBL}
PE-B\&B builds on the enumeration scheme proposed by \citet{LSS2004} and strengthens it with a tighter lower bound and two dominance rules. In this section, we present these enhancements and provide formal proofs of their correctness.

In PE-B\&B, job sequences are constructed from left to right: at each node $\sigma$, one child is generated for each job in $\overline{J}_{\sigma}$ by appending that job to the current partial sequence (see Figure~\ref{fig:BAB_schema}). This prefix-based branching structure makes it possible to exploit symmetries among consecutive jobs that can be processed without requiring any tool switches, retaining only the lexicographically ordered representative of each equivalence class. The PE-B\&B tree is explored depth-first, and the children are considered according to a fixed reference order, defined by the incumbent solution when available and by the original input order otherwise.

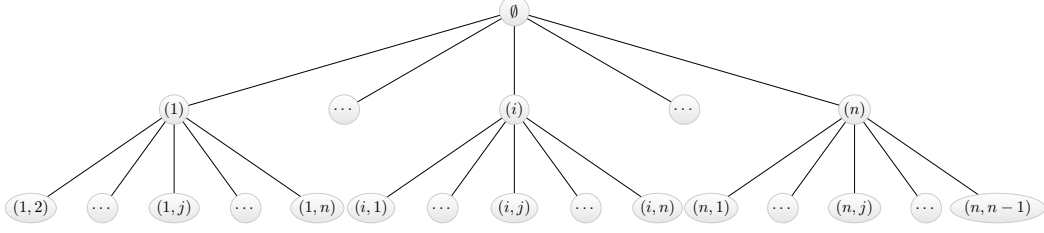
\begin{figure}
\centering
\resizebox{0.85\textwidth}{!}{
\tikzset{
    thick,
    tree node/.style = {align=center, inner sep=0pt, font=\small},
    every label/.append style = {font=\small},
    S/.style = {
        draw=black!20,
        ellipse,
        minimum size=7mm,
        inner sep=0pt,
        top color=white,
        bottom color=black!10
    },
    ENL/.style = {font=\small, left=1pt},
    ENR/.style = {font=\small, right=1pt},
    grow = down,
    level 1/.style = {sibling distance=3.8cm},
    level 2/.style = {sibling distance=1.6cm},
    level distance = 2.2cm
}
\centering
\begin{tikzpicture}[scale=0.85, every node/.style={scale=0.8}]
\node[S] {$\emptyset$}
    child {node[S] {$(1)$}
        child {node[S] {$(1,2)$}
            edge from parent node[ENL] {}
        }
        child {node[S] {$\dots$}
            edge from parent node[ENL] {}
        }
        child {node[S] {$(1,j)$}
            edge from parent node[ENL] {}
        }
        child {node[S] {$\dots$}
            edge from parent node[ENL] {}
        }
        child {node[S] {$(1,n)$}
            edge from parent node[ENR] {}
        }
        edge from parent node[ENL] {}
    }
    child {node[S] {$\dots$}
        edge from parent node[ENL] {}
    }
    child {node[S] {$(i)$}
        child {node[S] {$(i,1)$}
            edge from parent node[ENL] {}
        }
        child {node[S] {$\dots$}
            edge from parent node[ENL] {}
        }
        child {node[S] {$(i,j)$}
            edge from parent node[ENL] {}
        }
        child {node[S] {$\dots$}
            edge from parent node[ENL] {}
        }
        child {node[S] {$(i,n)$}
            edge from parent node[ENR] {}
        }
        edge from parent node[ENL] {}
    }
    child {node[S] {$\dots$}
        edge from parent node[ENR] {}
    }
    child {node[S] {$(n)$}
        child {node[S] {$(n,1)$}
            edge from parent node[ENL] {}
        }
        child {node[S] {$\dots$}
            edge from parent node[ENL] {}
        }
        child {node[S] {$(n,j)$}
            edge from parent node[ENL] {}
        }
        child {node[S] {$\dots$}
            edge from parent node[ENL] {}
        }
        child {node[S] {$(n,n-1)$}
            edge from parent node[ENR] {}
        }
        edge from parent node[ENR] {}
    };
\end{tikzpicture}
}
\caption{PE-B\&B tree.}
\label{fig:BAB_schema}
\end{figure}

\subsection{Lower Bounding Function \texorpdfstring{$L_4(\sigma)$}{L4(sigma)}}\label{Sec:LB4}
At a node $\sigma=(\sigma_1,\dots,\sigma_l)$ of the PE-B\&B search tree, the KTNS policy is used to compute an optimal sequence of magazine configurations $\mathcal{M}(\sigma)=(M_1,\dots,M_l)$ and its objective value $z(\sigma)=z(\mathcal{M}(\sigma))$. The objective value and the magazine configurations are then used to evaluate the new lower bound $L_4(\sigma)$ and apply the dominance rules, respectively.

The bound $L(\sigma)$ proposed by \citet{LSS2004} accounts only for the tools required by the last sequenced job $\sigma_l$ and by the unsequenced jobs, namely, the tools in $A_{\sigma}=T_{\sigma_l}\cup\left(\bigcup_{j\in\overline{J}_{\sigma}}T_j\right)$. In our notation, this bound can be written as $L(\sigma)=z(\sigma)+|A_{\sigma}|-c$.
By contrast, the new proposed bound also accounts for the tools in $T_{\sigma}\setminus A_{\sigma}$ that are already present in the magazine when $\sigma_l$ is processed. This refinement yields the following result.
\begin{prop}
A valid lower bound on the optimal value of the subproblem associated with node $\sigma$ of the PE-B\&B tree is $
L_4(\sigma)=z(\sigma)+|A_{\sigma}|-\min\{c,s_{\sigma}\}$,
where $
s_{\sigma}=\left|T_{\sigma}\cap A_{\sigma}\right|+\max\left\{c-\left|T_{\sigma}\right|,0\right\}$.
\end{prop}
\begin{proof}
Let $\overline{\sigma}$ be an arbitrary completion of $\sigma$. In $\overline{\sigma}$, every tool in $A_\sigma$ that is not present in the magazine while $\sigma_l$ is processed will contribute one additional tool switch beyond those already accounted for by $z(\sigma)$. It therefore remains to show that the number of tools in $A_\sigma$ that may already be present at that point is at most $\min\{c,s_\sigma\}$.
These tools may include those in $T_{\sigma}\cap A_\sigma$ together with at most $\max\{c-|T_{\sigma}|,0\}$ additional tools in $A_\sigma\setminus T_{\sigma}$ occupying free magazine slots. Hence, since the magazine capacity is $c$, at most $\min\{c,s_\sigma\}$ tools in $A_\sigma$ may be present while $\sigma_l$ is processed, which completes the proof.
\end{proof}

\subsection{Dominance Rules}\label{sec:BBL_symmetry}
The prefix structure of PE-B\&B makes it possible to reduce the number of symmetric solutions explored during the search. As observed by \citet{Akhundov2024}, one source of symmetry in the SSP arises when several jobs can be processed between two consecutive tool switches. In this case, different permutations of these jobs may yield solutions requiring the same number of tool switches. To break this symmetry, PE-B\&B keeps, among equivalent solutions, only the sequence in which consecutive interchangeable jobs are ordered lexicographically by their indices.

The first dominance rule directly implements this symmetry-breaking principle. In particular, when the last two jobs of a partial sequence are processed under the same magazine configuration, their order can be fixed lexicographically without excluding any potentially optimal solution.
\begin{prop}
Let $\sigma=(\sigma_1,\dots,\sigma_l)$ be a node of the PE-B\&B tree, and let $\mathcal{M}_{\sigma}=(M_1,\dots,M_l)$ be an optimal sequence of magazine configurations covering $\sigma$, computed according to the KTNS policy. Suppose that $M_{l-1}=M_l$ and $\sigma_{l-1}>\sigma_l$. Then, node $\sigma$ is dominated by $\sigma'=(\sigma_1,\dots,\sigma_{l-2},\sigma_l,\sigma_{l-1})$.
\end{prop}
\begin{proof}
Consider any completion $\overline{\sigma}=(\sigma_1,\dots,\sigma_l,\sigma_{l+1},\dots,\sigma_n)$ of $\sigma$. To prove the result, we show that there exists a corresponding completion $\overline{\sigma}'$ of $\sigma'$ such that $z(\overline{\sigma}')\leq z(\overline{\sigma})$.

Let $\mathcal{M}_{\overline{\sigma}}=(\overline{M}_1,\dots,\overline{M}_n)$ be the sequence of magazine configurations constructed by KTNS for $\overline{\sigma}$. Since $M_{l-1}=M_l$, all tools required by $\sigma_l$ can be kept in the magazine while processing $\sigma_{l-1}$. This property is preserved in $\overline{\sigma}$. Indeed, in any such completion, the tools required by $\sigma_l$ still have priority over tools required only by later jobs. Therefore, KTNS keeps these tools loaded while processing $\sigma_{l-1}$, and $\overline{M}_{l-1}=\overline{M}_l$.

Now consider the complete sequence $\overline{\sigma}'=(\sigma_1,\dots,\sigma_{l-2},\sigma_l,\sigma_{l-1}, \sigma_{l+1},\dots,\sigma_n)$, which belongs to the subtree rooted at $\sigma'$. Since $\overline{M}_{l-1}=\overline{M}_l$, $\mathcal{M}_{\overline{\sigma}}$ also covers $\overline{\sigma}'$. Thus, $z(\overline{\sigma}')\leq z(\mathcal{M}_{\overline{\sigma}}) =z(\overline{\sigma})$.
Hence, the optimal value of the subproblem associated with $\sigma'$ is not larger than that of the subproblem associated with $\sigma$. Therefore, $\sigma$ is dominated by $\sigma'$.
\end{proof}

The second dominance rule extends this symmetry argument to consecutive jobs processed under different magazine configurations. The two jobs can still be exchanged when the configuration associated with the latter also covers the former and every tool required by the former is needed again at or after the latter job in the current partial sequence.
\begin{prop}
Let $\sigma=(\sigma_1,\dots,\sigma_l)$ be a node of the PE-B\&B tree, and let $\mathcal{M}_{\sigma}=(M_1,\dots,M_l)$ be an optimal sequence of magazine configurations covering $\sigma$, computed according to the KTNS policy. Suppose that there exists an index $p\in\{2,\dots,l\}$ such that $T_{\sigma_{p-1}}\subseteq M_p$, $T_{\sigma_{p-1}}\subseteq\bigcup_{k=p}^{l}T_{\sigma_k}$, and $\sigma_{p-1}>\sigma_p$. Then, node $\sigma$ is dominated by $\sigma'=(\sigma_1,\dots,\sigma_{p-2},\sigma_p,\sigma_{p-1}, \sigma_{p+1},\dots,\sigma_l)$.
\end{prop}
\begin{proof}
Consider any completion $\overline{\sigma}=(\sigma_1,\dots,\sigma_l,\sigma_{l+1},\dots,\sigma_n)$ of $\sigma$. To prove the result, we show that there exists a corresponding completion $\overline{\sigma}'$ of $\sigma'$ such that $z(\overline{\sigma}')\leq z(\overline{\sigma})$.

Let $\mathcal{M}_{\overline{\sigma}}=(\overline{M}_1,\dots,\overline{M}_n)$ be the corresponding sequence of magazine configurations constructed by KTNS for $\overline{\sigma}$. Since $T_{\sigma_{p-1}}\subseteq M_p$, all tools required by $\sigma_{p-1}$ remain loaded while $\sigma_p$ is processed. The same holds in any completion $\overline{\sigma}$. Indeed, each of these tools is required again by some job between positions $p$ and $l$, and therefore has priority under KTNS over any tool required only after position $l$. Hence, appending jobs after $\sigma_l$ cannot cause any of the tools in $T_{\sigma_{p-1}}$ to be removed before $\sigma_p$ is processed, and therefore $T_{\sigma_{p-1}}\subseteq\overline{M}_p$.

Now consider the complete sequence $\overline{\sigma}'=(\sigma_1,\dots,\sigma_{p-2},\sigma_p,\sigma_{p-1}, \sigma_{p+1},\dots,\sigma_n)$, which belongs to the subtree rooted at $\sigma'$. Let $\mathcal{M}_{\overline{\sigma}'}$ be the sequence obtained from $\mathcal{M}_{\overline{\sigma}}$ by replacing $\overline{M}_{p-1}$ with $\overline{M}_p$, and leaving all other configurations unchanged. Observe that $\mathcal{M}_{\overline{\sigma}'}$ covers $\overline{\sigma}'$, and that $z(\mathcal{M}_{\overline{\sigma}'})\leq z(\mathcal{M}_{\overline{\sigma}})$, since
$
|\overline{M}_{p}\setminus\overline{M}_{p-2}|+|\overline{M}_{p}\setminus\overline{M}_{p}|\leq|\overline{M}_{p-1}\setminus\overline{M}_{p-2}|+|\overline{M}_{p}\setminus\overline{M}_{p-1}|$, 
while all other transitions are unchanged. Thus, $z(\overline{\sigma}')\leq z(\overline{\sigma})$, which proves that the optimal value of the subproblem associated with $\sigma'$ is not larger than that of the subproblem associated with $\sigma$. Therefore, $\sigma$ is dominated by $\sigma'$.
\end{proof}

\section{The PI-B\&B-H Algorithm}\label{sec:Heuristic}
In this section, we introduce PI-B\&B-H, which is invoked during the preprocessing phase (see Section~\ref{sec:ExactAlgorithm}) to compute an initial high-quality upper bound within a limited computational time. PI-B\&B-H is derived from PI-B\&B, but differs from it in three main aspects: it uses a different search strategy, includes a completion pruning rule, and applies an aggressive fathoming rule to discard branches unlikely to lead to high-quality solutions. These three components are described in the remainder of this section.

The different search strategy and the completion pruning rule do not compromise the exactness of the algorithm; rather, they are introduced to find a high-quality incumbent solution more quickly. These features are not implemented in PI-B\&B, since a high-quality incumbent solution is already provided by the preprocessing phase. In contrast to these two components, the additional fathoming rule may discard branches that are unlikely to lead to high-quality solutions. Therefore, when this fathoming rule is applied, some feasible solutions may be discarded during the search, and the resulting method is no longer exact.

Unlike PI-B\&B, which computes only the value $z(\sigma)$ at each node $\sigma$ using IGA (see Section~\ref{sec:LB}), the three novel features of PI-B\&B-H also require an optimal sequence of magazine configurations $\mathcal{M}_{\sigma}$, which is computed using the KTNS policy.

\subsection{Search Strategy}
PI-B\&B-H also follows a depth-first search strategy. At a node $\sigma=(\sigma_1,\dots,\sigma_l)$, if the node is not fathomed by the bounding or fathoming procedures, a job $j'\in\overline{J}_{\sigma}$ is selected according to the rule described in Section~\ref{sec:branching_scheme}, and all $l+1$ child nodes obtained by inserting $j'$ into $\sigma$ are generated. Among these child nodes, PI-B\&B-H first explores the one obtained by inserting $j'$ after the earliest position in $\operatorname*{arg\,min}_{i\in\{0,\dots,l\}} \left\{\min\left\{|M_i\cup T_{j'}|,\ |M_{i+1}\cup T_{j'}|\right\}\right\}$, where $M_0=M_{l+1}=T$. This choice favors insertion positions where the selected job is compatible with at least one of the neighboring magazine configurations, and is therefore expected to limit the number of additional tool switches.

\subsection{Completion Pruning Rule}
The completion pruning rule allows PI-B\&B-H to prune a node $\sigma$ by identifying a complete solution in the subtree rooted at $\sigma$ whose cost equals that of the current partial sequence, $z(\sigma)$.
Suppose that each job $j\in\overline{J}_{\sigma}$ is covered by at least one configuration in $\mathcal{M}_{\sigma}$. Then, all unsequenced jobs can be inserted into $\sigma$ without increasing the number of tool switches. In particular, a complete sequence $\sigma^*$ satisfying $z(\sigma^*)=z(\sigma)$ can be constructed by iteratively inserting each job $j\in\overline{J}_{\sigma}$ next to a sequenced job whose associated magazine configuration covers $j$. Since $z(\sigma)$ is a lower bound on the cost of every completion of $\sigma$, $\sigma^*$ is optimal among all such completions. The incumbent can therefore be updated with $\sigma^*$, and node $\sigma$ can be pruned.

\subsection{An Almost Exact Fathoming Rule}\label{sec:nonexactrule}
The following fathoming rule is introduced to reduce the number of symmetric solutions explored by PI-B\&B-H, and is based on an argument similar to the one used in Section~\ref{sec:BBL_symmetry}.
Specifically, let $\sigma=(\sigma_1,\dots,\sigma_l)$ be a node of the PI-B\&B-H tree, and let $\mathcal{M}_{\sigma}=(M_1,\dots,M_l)$ be an optimal sequence of magazine configurations covering $\sigma$. Suppose that there exists an index $i\in\{1,\dots,l-1\}$ such that $T_{\sigma_i}\subseteq M_{i+1}$ and $\sigma_i>\sigma_{i+1}$. Then, $\mathcal{M}_{\sigma}$ also covers the sequence $\sigma'=(\sigma_1,\dots,\sigma_{i-1},\sigma_{i+1},\sigma_i,\sigma_{i+2},\dots,\sigma_l)$, obtained by swapping the two consecutive jobs $\sigma_i$ and $\sigma_{i+1}$. Hence, $z(\sigma')\leq z(\sigma)$. Therefore, PI-B\&B-H fathoms node $\sigma$, avoiding the exploration of a branch that is locally dominated by $\sigma'$. However, this rule is not exact, since the local dominance relation between $\sigma$ and $\sigma'$ does not necessarily extend to all their possible completions.

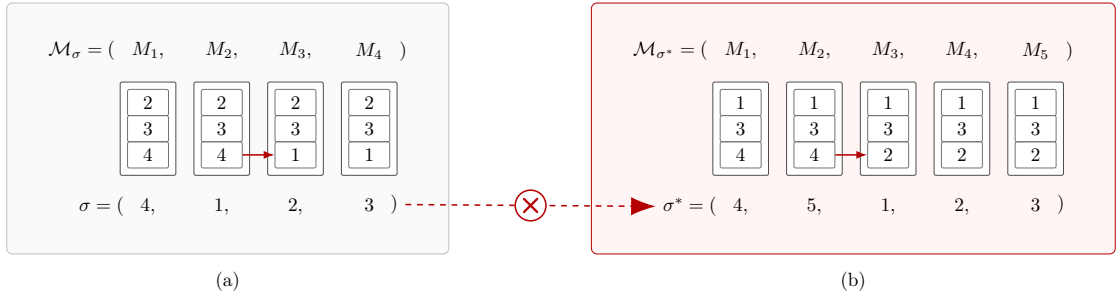
\begin{figure}
    \centering
    \resizebox{0.9\textwidth}{!}{
    \begin{tikzpicture}[
        font=\small,
        mag/.style={
            matrix of nodes,
            nodes={
                draw=black!55,
                minimum width=0.75cm,
                minimum height=0.42cm,
                anchor=center,
                font=\small
            },
            column sep=-\pgflinewidth,
            row sep=-\pgflinewidth,
            draw=black!70,
            rounded corners=1pt,
            fill=white
        },
        panel/.style={
            draw=black!25,
            rounded corners=3pt,
            fill=black!2,
            inner xsep=18pt,
            inner ysep=16pt
        },
        pruned/.style={
            draw=red!70!black,
            rounded corners=3pt,
            fill=red!4,
            inner xsep=18pt,
            inner ysep=16pt
        },
        transition/.style={
            -{Latex[length=2mm]},
            thick,
            draw=black!60
        },
        badtransition/.style={
            -{Latex[length=2mm]},
            thick,
            draw=red!70!black
        }
    ]
    \def\xstep{1.35}
    \def\xgap{5.60}
    \def\ybase{0}
    \matrix[mag] (mat1) at (0,\ybase) {2\\3\\4\\};
    \matrix[mag] (mat2) at (\xstep,\ybase) {2\\3\\4\\};
    \matrix[mag] (mat3) at (2*\xstep,\ybase) {2\\3\\1\\};
    \matrix[mag] (mat4) at (3*\xstep,\ybase) {2\\3\\1\\};
    \draw[badtransition]
    (mat2-3-1.east) to[out=0,in=180] (mat3-3-1.west);

    \node[above=8pt of mat1] (lm1) {$M_1,$};
    \node[above=8pt of mat2] (lm2) {$M_2,$};
    \node[above=8pt of mat3] (lm3) {$M_3,$};
    \node[above=8pt of mat4] (lm4) {$M_4$};
    \node[left=1pt of lm1] (lmprefix) {$\mathcal{M}_{\sigma}=($};
    \node[right=1pt of lm4] (lmsuffix) {$)$};
    \node[below=8pt of mat1] (ls1) {$4,$};
    \node[below=8pt of mat2] (ls2) {$1,$};
    \node[below=8pt of mat3] (ls3) {$2,$};
    \node[below=8pt of mat4] (ls4) {$3$};
    \node[left=1pt of ls1] (lsprefix) {$\sigma=($};
    \node[right=1pt of ls4] (lssuffix) {$)$};
    \begin{scope}[on background layer]
        \node[
            panel,
            fit=(lmprefix)(lmsuffix)(mat1)(mat4)(lsprefix)(lssuffix)
        ] (leftpanel) {};
    \end{scope}
    \node[below=5pt of leftpanel] {(a)};
    \def\xstartR{3*\xstep+\xgap+1.2}
    \matrix[mag] (mat5) at (\xstartR,\ybase) {1\\3\\4\\};
    \matrix[mag] (mat6) at (\xstartR+\xstep,\ybase) {1\\3\\4\\};
    \matrix[mag] (mat7) at (\xstartR+2*\xstep,\ybase) {1\\3\\2\\};
    \matrix[mag] (mat8) at (\xstartR+3*\xstep,\ybase) {1\\3\\2\\};
    \matrix[mag] (mat9) at (\xstartR+4*\xstep,\ybase) {1\\3\\2\\};
        \draw[badtransition]
    (mat6-3-1.east) to[out=0,in=180] (mat7-3-1.west);

    \node[above=8pt of mat5] (rm1) {$M_1,$};
    \node[above=8pt of mat6] (rm2) {$M_2,$};
    \node[above=8pt of mat7] (rm3) {$M_3,$};
    \node[above=8pt of mat8] (rm4) {$M_4,$};
    \node[above=8pt of mat9] (rm5) {$M_5$};
    \node[left=1pt of rm1] (rmprefix) {$\mathcal{M}_{\sigma^*}=($};
    \node[right=1pt of rm5] (rmsuffix) {$)$};
    \node[below=8pt of mat5] (rs1) {$4,$};
    \node[below=8pt of mat6] (rs2) {$5,$};
    \node[below=8pt of mat7] (rs3) {$1,$};
    \node[below=8pt of mat8] (rs4) {$2,$};
    \node[below=8pt of mat9] (rs5) {$3$};
    \node[left=1pt of rs1] (rsprefix) {$\sigma^*=($};
    \node[right=1pt of rs5] (rssuffix) {$)$};
    \begin{scope}[on background layer]
        \node[
            pruned,
            fit=(rmprefix)(rmsuffix)(mat5)(mat9)(rsprefix)(rssuffix)
        ] (rightpanel) {};
    \end{scope}
    \node[below=5pt of rightpanel] {(b)};
\path (lssuffix.east) -- (rsprefix.west)
    coordinate[pos=0.47] (blockL)
    coordinate[pos=0.53] (blockR)
    coordinate[pos=0.50] (blockC);
\draw[-, thick, dashed, draw=red!70!black]
    (lssuffix.east) -- (blockL);
\draw[-{Latex[length=4.5mm]}, thick, dashed, draw=red!70!black]
    (blockR) -- (rsprefix.west);
\node[
    circle,
    draw=red!70!black,
    line width=0.8pt,
    fill=white,
    minimum size=16pt,
    inner sep=0pt
] (blockmark) at (blockC) {};
\draw[red!70!black, line width=1.2pt, line cap=round]
    ($(blockmark.center)+(-0.12,0.12)$) --
    ($(blockmark.center)+(0.12,-0.12)$);
\draw[red!70!black, line width=1.2pt, line cap=round]
    ($(blockmark.center)+(-0.12,-0.12)$) --
    ($(blockmark.center)+(0.12,0.12)$);
    \end{tikzpicture}
    }
\caption{
(a) Optimal magazine configurations $\mathcal{M}_{\sigma}$ covering $\sigma=(4,1,2,3)$. Since $T_4\subseteq M_2$ and
$\sigma_1>\sigma_2$, node $\sigma$ is pruned.
(b) Optimal magazine configurations for the optimal solution
$\sigma^*=(4,5,1,2,3)$, which cannot be generated after pruning
$\sigma$.}
    \label{fig:non_exact_rule}
\end{figure}

As an example, consider an instance with $n=5$, $m=4$, and $c=3$, where $T_1=\{2,3\}$, $T_2=\{1,3\}$, $T_3=\{1,2\}$, $T_4=\{3,4\}$, and $T_5=\{1,4\}$. An optimal solution is $\sigma^*=(4,5,1,2,3)$, with $z(\sigma^*)=m-c=1$.
Assume that PI-B\&B-H reaches node $(1,2,3)$. In order to generate $\sigma^*$ from this node, job $4$ must first be inserted before job $1$, leading to the partial sequence $\sigma=(4,1,2,3)$. For this sequence, an optimal sequence of magazine configurations obtained by applying the KTNS policy satisfies $T_4\subseteq M_2$, since jobs $4$ and $1$ are covered by the same magazine configuration $\{2,3,4\}$; see Figure~\ref{fig:non_exact_rule}(a). Moreover, the two corresponding jobs satisfy $\sigma_1>\sigma_2$. Therefore, according to the fathoming rule, node $\sigma$ is pruned.
As a consequence, PI-B\&B-H cannot insert job $5$ between jobs $4$ and $1$, and the optimal completion $\sigma^*=(4,5,1,2,3)$ cannot be generated; see Figure~\ref{fig:non_exact_rule}(b). Observe that the arguments used in Section~\ref{sec:BBL_symmetry} to prove the correctness of the dominance rules do not apply here.
Indeed, consider the locally dominating node $\sigma'=(1,4,2,3)$ associated with the fathomed node $\sigma$. The completion $\overline{\sigma}'=(5,1,4,2,3)$ of $\sigma'$, which provides the natural counterpart to the optimal completion $\sigma^*=(4,5,1,2,3)$ of $\sigma$, satisfies $z(\overline{\sigma}')=2>1=z(\sigma^*)$. Thus, the local dominance relation between $\sigma$ and $\sigma'$ is not necessarily preserved after the insertion of the remaining jobs.

This example shows that the rule may prune a branch containing an optimal completion. However, this does not necessarily imply that PI-B\&B-H fails to find an optimal solution for the instance, since another optimal solution may still be generated through a different, non-pruned branch. In fact, in almost all tested instances, PI-B\&B-H still finds an optimal solution, while suboptimal solutions are obtained only in very few cases.

\section{An Exact Algorithm for the SSP}\label{sec:ExactAlgorithm}
This section presents C-B\&B, an exact algorithm for the SSP that integrates the B\&B algorithms introduced in the previous sections. An informal pseudo-code of C-B\&B is reported in Algorithm~\ref{alg:ExactAlgorithm}. The algorithm first computes the initial lower bound $lb$ as $L_3(\emptyset)$, and initializes the incumbent upper bound $ub$ to a sufficiently large value, as shown in line~\ref{alg:Exact_Step1}.

A short preliminary run of PE-B\&B is then performed with time limit $t_{\max}$ (see line~\ref{alg:Exact_Step2}). This step is mainly intended to detect easy instances, and in particular those whose optimal value is equal to the trivial lower bound $m-c$. After this run, the incumbent sequence, the lower bound and the upper bound are updated in line~\ref{alg:Exact_Step3}. If $lb=ub$, the incumbent solution is certified to be optimal and the algorithm returns it, as shown in line~\ref{alg:Exact_Step4}. Otherwise, PI-B\&B-H is invoked in line~\ref{alg:Exact_Step5} with the same short time limit $t_{\max}$. Line~\ref{alg:Exact_Step6} updates the incumbent solution and the upper bound with the best sequence found by PI-B\&B-H. The same line updates $lb$ with the lower bound it has proved, but only if PI-B\&B-H has not applied the non-exact fathoming rule described in Section~\ref{sec:nonexactrule}. If $lb=ub$, the incumbent solution is certified to be optimal and returned, as shown in line~\ref{alg:Exact_Step7}. If PI-B\&B-H reaches the time limit, HGS is invoked in line~\ref{alg:Exact_Step9}. In this phase, we rely on an enriched implementation of the HGS, which differs from the original one in two aspects. First, each individual generated during the genetic search is evaluated using the IGA instead of the KTNS procedure originally adopted, thus considerably reducing the total evaluation time. Second, the incumbent sequence $\sigma^*$ found by PI-B\&B-H is inserted into the initial population to speed up the convergence of the algorithm, while the remaining individuals are generated randomly to preserve diversity. Once the HGS terminates, the upper bound is updated with the best solution value it has returned, as shown in line~\ref{alg:Exact_Step10}. If this update makes $ub$ equal to $lb$, the incumbent is again certified to be optimal and returned.

The algorithm then enters the final exact enumeration phase. The exact solver, either PI-B\&B or PE-B\&B, is selected according to the remaining gap with respect to the trivial lower bound $m-c$. If $ub-(m-c)\leq \Delta$, PE-B\&B is invoked (see line~\ref{alg:Exact_Step12}); otherwise, PI-B\&B is executed (see line~\ref{alg:Exact_Step13}). The parameter $\Delta$ therefore determines the threshold used to select between the two exact solvers. This choice is motivated by the observation that, when the gap is small, only a few tool switches beyond the trivial lower bound remain possible. In this case, several jobs are likely to be processed under identical magazine configurations, making the dominance rules introduced in Section~\ref{sec:BBL_symmetry} particularly effective. Conversely, when the gap is larger, the branching scheme of PI-B\&B tends to be more effective, as it promotes a faster increase in the lower bounds associated with the explored nodes. The selected exact algorithm is then executed until optimality is proven, updating the incumbent sequence and the upper bound whenever a better solution is found. Finally, C-B\&B returns the incumbent sequence $\sigma^*$, together with its value $ub$, as shown in line~\ref{alg:Exact_Step15}. 
\begin{algorithm}
\small
\caption{Exact Combinatorial B\&B Algorithm for the SSP}
\label{alg:ExactAlgorithm}
\KwIn{An SSP instance, $t_{\max}$, $\Delta$}
$\sigma^* \leftarrow \emptyset$; $lb \leftarrow L_3(\emptyset)$; $ub \leftarrow +\infty$\;\label{alg:Exact_Step1}
Run \textsc{PE-B\&B} with time limit $t_{\max}$\;\label{alg:Exact_Step2}
Update $\sigma^*$, $lb$, and set $ub \leftarrow z(\sigma^*)$\;\label{alg:Exact_Step3}
\lIf{$lb=ub$}{\Return{$\sigma^*$, $ub$}}\label{alg:Exact_Step4}
Run \textsc{PI-B\&B-H} with time limit $t_{\max}$\;\label{alg:Exact_Step5}
Update $\sigma^*$, $lb$, and set $ub \leftarrow z(\sigma^*)$\;\label{alg:Exact_Step6}
\lIf{$lb=ub$}{\Return{$\sigma^*$, $ub$}}\label{alg:Exact_Step7}
\If{\textsc{PI-B\&B-H} reaches the time limit}{
   Run HGS\;\label{alg:Exact_Step9}
   Update $\sigma^*$ and set $ub \leftarrow z(\sigma^*)$\;\label{alg:Exact_Step10}
   \lIf{$lb=ub$}{\Return{$\sigma^*$, $ub$}}\label{alg:Exact_Step11}
}
\eIf{$(ub-(m-c))\leq \Delta$}{
   Run \textsc{PE-B\&B}\;\label{alg:Exact_Step12}
}{
   Run \textsc{PI-B\&B}\;\label{alg:Exact_Step13}
}
Update $\sigma^*$ and set $ub \leftarrow z(\sigma^*)$\;\label{alg:Exact_Step14}
\Return{$\sigma^*$, $ub$}\;\label{alg:Exact_Step15}
\end{algorithm}
\section{Computational Results}\label{sec:ComputationalResults}
In this section, we assess the performance of the proposed C-B\&B algorithm on the benchmark instances from the literature and compare it with the best exact approaches  available for the SSP. 
All algorithms introduced in the previous sections were implemented in C++ and run on an Apple MacBook Air (M3) with 16 GB of RAM, under macOS 26. Algorithm C-B\&B was given a time limit of 1200 seconds per instance. Its two input parameters, $\Delta$ and $t_{\max}$, were calibrated by means of preliminary experiments on the whole benchmark and set to 4 and $0.001$ seconds, respectively.

\subsection{Benchmark Instances}\label{sec:instances}
We consider the three sets of benchmark instances commonly used in the SSP literature to evaluate exact algorithms. All of them are publicly available at \href{https://github.com/jordanamecler/HGS-SSP/tree/master/Instances}{\nolinkurl{https://github.com/jordanamecler/HGS-SSP/tree/master/Instances}}, each in a separate folder: ``Crama'', ``Yanasse'', and ``Catanzaro'', containing the instances proposed by \citet{CKOS1994}, \citet{Yanasse2009}, and \citet{Catanzaro2015766}, respectively. All the instances were produced by means of the instance-generating scheme introduced by \citet{CKOS1994}, which takes as input the parameters $n$, $m$, $c$, and the minimum and maximum number of tools that a job can require, denoted by ${Min}$ and ${Max}$, respectively. Given a combination of these parameters, the number of tools required by each job is randomly drawn from the interval $[{Min},{Max}]$, and the corresponding toolset is then randomly generated so that it neither contains nor is contained in any of the toolsets generated before it, ensuring that no job is dominated.

Table~\ref{tab:instances} summarizes how each set is partitioned into groups, each collecting the instances generated under a different parameter configuration, and reports, for each group, the corresponding parameter values (or the intervals they range in), together with the number of instances it contains (column ``\#Inst''). 
Following the nomenclature adopted by \citet{Mecler2021}, we denote the four groups of the first set by $C_1$, $C_2$, $C_3$, and $C_4$, the five groups of the second set by $A$, $B$, $C$, $D$, and $E$, and the four groups of the third set by ${datA}$, ${datB}$, ${datC}$, and ${datD}$.
Note that the three sets partly overlap in the configurations they cover, as both \citet{Yanasse2009} and \citet{Catanzaro2015766} reused parameter combinations already adopted by \citet{CKOS1994}: groups ${datA}$, ${datB}$, ${datC}$, and ${datD}$ stem from exactly the same parameter combinations as $C_1$, $C_2$, $C_3$, and $C_4$, respectively, whereas group $E$ is split into two halves, covering the combinations of $C_1$ and $C_2$, respectively.

\subsection{Comparison with the Existing Literature}\label{sec:comparison}
We compare C-B\&B with the three most effective exact approaches currently available for the SSP, namely:
\begin{itemize}
\item LSS04, the B\&B algorithm by \citet{LSS2004};
\item SCY21, the multicommodity flow formulation by \citet{daSilva2021};
\item AO24, the ILP approach by \citet{Akhundov2024}.
\end{itemize}
The results of LSS04 were obtained by re-implementing the algorithm described by \citet{LSS2004}, including both lower bounds they proposed. Our implementation was run in the same computational environment and under the same time limit as C-B\&B and, to make the comparison as fair as possible, it was also provided with the initial upper bound computed by HGS, which we improved by evaluating each individual of the genetic search with the IGA instead of the KTNS procedure originally adopted. 
The results of SCY21 and AO24, in contrast, were taken directly from \citet{Akhundov2024}, where the corresponding models were solved with Gurobi 8.0.1, through its Python interface, on a workstation equipped with an AMD Ryzen Threadripper 2950X 16-core processor and 64 GB of RAM, running Ubuntu 18.04.6 LTS, with a time limit of 3600 seconds per instance.
Although that environment differs from the one in which LSS04 and C-B\&B were run, it is worth noting that SCY21 and AO24 were granted a time limit three times larger than that allowed to our algorithm, and that the margin by which C-B\&B outperforms them is far too wide to be ascribed to the difference in the hardware adopted.

\begin{table}
\caption{Overview of the SSP benchmark instances and summary of the computational results.}
\label{tab:instances}
\centering
\setlength{\tabcolsep}{3pt}
\begin{adjustbox}{max width=\textwidth}
\begin{tabular}{llcccccrrrrrrrrr}
\toprule
 & & & & & & & & \multicolumn{2}{c}{LSS04} & \multicolumn{2}{c}{SCY21}
                & \multicolumn{2}{c}{AO24} & \multicolumn{2}{c}{C-B\&B} \\
\cmidrule(lr){9-10} \cmidrule(lr){11-12} \cmidrule(lr){13-14} \cmidrule(lr){15-16}
Set & Group & $n$ & $m$ & ${Min}$ & ${Max}$ & $c$ & \#Inst & Opt & Sec & Opt & Sec & Opt & Sec & Opt & Sec \\
\cmidrule(lr){1-8} \cmidrule(lr){9-16}
\multirow{4}{*}{\citet{CKOS1994}}
 & $C_1$ & 10 & 10 & 2 & 4  & $[4,7]$   & 40 & \textbf{40} & 0.02 & -- & & -- & &  \textbf{40} & \textbf{0.00} \\
 & $C_2$ & 15 & 20 & 2 & 6  & $[6,12]$  & 40 & \textbf{40} & 0.85 & -- & & -- & & \textbf{40} & \textbf{0.05} \\
 & $C_3$ & 30 & 40 & 5 & 15 & $[15,25]$ & 40 & 0 & & -- & & -- & & \textbf{3} & \textbf{648.32} \\
 & $C_4$ & 40 & 60 & 7 & 20 & $[20,30]$ & 40 & -- & & -- & & -- & & -- & \\
\cmidrule(lr){2-8} \cmidrule(lr){9-16}
 & \multicolumn{6}{l}{\textbf{Total}} & {160} & 80 & & -- & & -- & & \textbf{83} & \\
\cmidrule(lr){1-8} \cmidrule(lr){9-16}
\multirow{5}{*}{\citet{Yanasse2009}}
 & $A$ & 8         & $[15,25]$ & $[2,15]$ & $[5,20]$ & $[5,20]$ & 340 & \textbf{340} & 0.01   & \textbf{340} &        & \textbf{340} &        & \textbf{340} & \textbf{0.00}  \\
 & $B$ & 9         & $[15,25]$ & $[2,15]$ & $[5,20]$ & $[5,20]$ & 370 & \textbf{370} & 0.02   & \textbf{370} & 5.16   & \textbf{370} & 19.12  & \textbf{370} & \textbf{0.00}  \\
 & $C$ & 15        & $[15,25]$ & $[2,15]$ & $[5,20]$ & $[5,20]$ & 340 &          339 & 46.68  &          227 & 788.05 &          263 & 504.05 & \textbf{340} & \textbf{0.07}  \\
 & $D$ & $[20,25]$ & $[15,25]$ & $[2,10]$ & $[5,15]$ & $[5,20]$ & 260 &          169 & 149.86 &          114 & 132.30 &          169 & 137.68 & \textbf{260} & \textbf{15.63} \\
 & $E$ & $[10,15]$ & $[10,20]$ & 2        & $[4,6]$  & $[4,12]$ &  80 &  \textbf{80} & 1.63   &           71 & 187.08 &           75 & 100.90 &  \textbf{80} & \textbf{0.02}  \\
\cmidrule(lr){2-8} \cmidrule(lr){9-16}
 & \multicolumn{6}{l}{\textbf{Total}} & {1390} &1298& & 1122 & & 1217 & & \textbf{1390} & \\
\cmidrule(lr){1-8} \cmidrule(lr){9-16}
\multirow{4}{*}{\citet{Catanzaro2015766}}
 & ${datA}$ & 10 & 10 & 2 & 4  & $[4,7]$   & 40 & \textbf{40} & 0.02 & -- & & -- & & \textbf{40} & \textbf{0.00} \\
 & ${datB}$ & 15 & 20 & 2 & 6  & $[6,12]$  & 40 & \textbf{40} & 1.29 & -- & & -- & & \textbf{40} & \textbf{0.05} \\
 & ${datC}$ & 30 & 40 & 5 & 15 & $[15,25]$ & 40 & 0 & & -- & & -- & & \textbf{1} & \textbf{943.05} \\
 & ${datD}$ & 40 & 60 & 7 & 20 & $[20,30]$ & 40 & -- & & -- & & -- & & -- & \\
\cmidrule(lr){2-8} \cmidrule(lr){9-16}
 & \multicolumn{6}{l}{\textbf{Total}} & {160} & 80 & & -- & & -- & & \textbf{81} & \\
\bottomrule
\end{tabular}
\end{adjustbox}
\end{table}

Table~\ref{tab:instances} reports the overall results obtained by each algorithm: columns ``Opt'' give the number of instances solved to proven optimality, whereas columns ``Sec'' give the average solution time, in seconds, computed over the solved instances. A dash in the ``Opt'' column means that no result is available for that algorithm on the corresponding group, whereas an empty entry in the ``Sec'' column means that the average computing time is either undefined, because no instance of the group was solved, or unknown, because the computing times are not available. 
C-B\&B is the only algorithm that closes the whole benchmark proposed by \citet{Yanasse2009}, commonly used to test exact algorithms for the SSP: it solves all 1390 instances, whereas LSS04, SCY21, and AO24 leave 92, 268, and 173 of them open, respectively. 
Moreover, C-B\&B is the fastest on every group by a large margin: on the instances with at most 15 jobs its average time never exceeds $0.07$ seconds, and no single instance takes more than $0.62$ seconds, against the $46.68$ seconds of LSS04 and the several hundred seconds of SCY21 and AO24. On group $D$, it closes all 260 instances in $15.63$ seconds on average, against the more than $130$ seconds that the other approaches spend on the instances they do solve.
On the other two sets, the results are in line with those discussed above: C-B\&B closes all the instances with 10 jobs in negligible time and all those with 15 jobs in $0.05$ seconds on average. 
Finally, C-B\&B is also able to close instances with 30 jobs: it solves 3 of the 40 instances of group $C_3$ and 1 of the 40 instances of group ${datC}$.  To the best of our knowledge, it is the first exact algorithm able to solve instances of this size to proven optimality.

\subsection{C-B\&B Component Analysis}\label{sec:component_analysis}
We now analyze the behavior of C-B\&B, and the contribution of its individual components, on the most difficult instances of the benchmark generated by \citet{Yanasse2009}, namely those of groups $C$ and $D$. As a baseline we adopt LSS04, which is the best performing of the three exact approaches from the literature. 
\begin{table}
\caption{C-B\&B component analysis and comparison with LSS04 on the instances of groups $C$ and $D$.}
\label{tab:laporteCD}
\centering
\setlength{\tabcolsep}{3pt}
\begin{adjustbox}{max width=\textwidth}
\begin{tabular}{ccrrrrrrrrrrrrrrrr}
\toprule
& & & & & \multicolumn{4}{c}{\textbf{LSS04}} & \multicolumn{9}{c}{\textbf{C-B\&B}} \\
\cmidrule(lr){6-9} \cmidrule(lr){10-18}
& & & & & & & & & \multicolumn{2}{c}{\textbf{Prep.}} & \multicolumn{2}{c}{\textbf{PI-B\&B}} & \multicolumn{2}{c}{\textbf{PE-B\&B}} & \multicolumn{3}{c}{\textbf{Overall}} \\
\cmidrule(lr){10-11} \cmidrule(lr){12-13} \cmidrule(lr){14-15} \cmidrule(lr){16-18}
Group & $n$ & $m$ & $c$ & \#Inst & Opt & $\mathrm{Sec}_{\mathrm{HGS}}$ & Sec & Nodes & Opt & $\mathrm{Sec}$ & Opt & Sec & Opt & Sec & Opt & Sec & Nodes \\
\cmidrule(lr){1-5} \cmidrule(lr){6-9} \cmidrule(lr){10-18}
\multirow{9}{*}{$C$}
 & 15 & 15 & 5  & 10  & \textbf{10}  & 0.06 & 5.53  & \num{12258859}  & 0   & 0.03 & 9   & 0.00 & 1   & 0.06 & \textbf{10}  & \textbf{0.04} & \textbf{56\,188}  \\
 & 15 & 15 & 10 & 30  & \textbf{30}  & 0.13 & 24.21 & \num{49390630}  & 10  & 0.03 & 14  & 0.01 & 6   & 0.04 & \textbf{30}  & \textbf{0.04} & \textbf{59\,328}  \\
 & 15 & 20 & 5  & 10  & \textbf{10}  & 0.06 & 2.68  & \num{6434640}   & 0   & 0.03 & 7   & 0.00 & 3   & 0.17 & \textbf{10}  & \textbf{0.08} & \textbf{183\,426} \\
 & 15 & 20 & 10 & 30  & \textbf{30}  & 0.12 & 32.99 & \num{47402296}  & 9   & 0.05 & 19  & 0.01 & 2   & 0.00 & \textbf{30}  & \textbf{0.06} & \textbf{32\,648}  \\
 & 15 & 20 & 15 & 60  & 59           & 0.18 & 43.86 & \num{63170190}  & 19  & 0.03 & 28  & 0.01 & 13  & 0.07 & \textbf{60}  & 0.05          & \num{94380}       \\
 & 15 & 25 & 5  & 10  & \textbf{10}  & 0.08 & 3.37  & \num{6653115}   & 1   & 0.06 & 8   & 0.01 & 1   & 0.14 & \textbf{10}  & \textbf{0.08} & \textbf{86\,963}  \\
 & 15 & 25 & 10 & 30  & \textbf{30}  & 0.14 & 58.74 & \num{77253535}  & 5   & 0.09 & 20  & 0.02 & 5   & 0.00 & \textbf{30}  & \textbf{0.10} & \textbf{95\,876}  \\
 & 15 & 25 & 15 & 60  & \textbf{60}  & 0.17 & 92.76 & \num{112383683} & 13  & 0.07 & 31  & 0.02 & 16  & 0.04 & \textbf{60}  & \textbf{0.09} & \textbf{175\,525} \\
 & 15 & 25 & 20 & 100 & \textbf{100} & 0.24 & 40.78 & \num{51919350}  & 36  & 0.03 & 38  & 0.02 & 26  & 0.09 & \textbf{100} & \textbf{0.06} & \textbf{135\,645} \\
\cmidrule(lr){2-5} \cmidrule(lr){6-9} \cmidrule(lr){10-18}
 & \multicolumn{3}{l}{\textbf{Tot/Avg}} & 340 & 339 & 0.17 & 46.68 & \num{62350611} & 93 & 0.05 & 174 & 0.01 & 73 & 0.07 & \textbf{340} & 0.07 & \num{113706} \\
\midrule
\multirow{15}{*}{$D$}
 & 20 & 15 & 5  & 10 & 1           & 0.22 & 975.19 & \num{2353705388} & 0  & 0.18 & 10 & 0.33  & 0  &       & \textbf{10} & 0.51          & \num{2816070}      \\
 & 20 & 15 & 10 & 20 & 11          & 0.62 & 172.56 & \num{417589662}  & 3  & 0.24 & 7  & 0.11  & 10 & 2.46  & \textbf{20} & 1.50          & \num{4627743}      \\
 & 20 & 20 & 5  & 10 & 6           & 0.20 & 428.80 & \num{951859253}  & 0  & 0.20 & 10 & 0.08  & 0  &       & \textbf{10} & 0.28          & \num{467471}       \\
 & 20 & 20 & 10 & 10 & \textbf{10} & 0.57 & 43.09  & \num{95388305}   & 4  & 0.42 & 0  &       & 6  & 0.26  & \textbf{10} & \textbf{0.58} & \textbf{511\,236}  \\
 & 20 & 20 & 15 & 30 & 23          & 0.85 & 211.39 & \num{459139137}  & 9  & 0.41 & 0  &       & 21 & 2.39  & \textbf{30} & 2.08          & \num{5006064}      \\
 & 20 & 25 & 5  & 10 & 5           & 0.19 & 306.95 & \num{640293119}  & 0  & 0.22 & 9  & 0.16  & 1  & 25.72 & \textbf{10} & 2.94          & \num{8955572}      \\
 & 20 & 25 & 10 & 10 & \textbf{10} & 0.53 & 8.61   & \num{16853431}   & 6  & 0.58 & 0  &       & 4  & 0.37  & \textbf{10} & \textbf{0.73} & \textbf{421\,612}  \\
 & 20 & 25 & 15 & 40 & 25          & 0.90 & 201.87 & \num{349631862}  & 10 & 0.57 & 21 & 41.05 & 9  & 7.27  & \textbf{40} & 23.76         & \num{187386102}    \\
 & 20 & 25 & 20 & 40 & 31          & 1.19 & 85.06  & \num{156846267}  & 20 & 0.42 & 1  & 0.02  & 19 & 8.13  & \textbf{40} & 4.28          & \num{9080667}      \\
\cmidrule(lr){2-5} \cmidrule(lr){6-9} \cmidrule(lr){10-18}
  &   \multicolumn{2}{l}{\textbf{Tot/Avg}} & & 180 & 122 & 0.81 & 164.29 & \num{337257878} & 52 & 0.40 & 58 & 14.97 & 70 & 4.62 & \textbf{180} & 7.02 & \num{45739595} \\
\cmidrule(lr){2-5} \cmidrule(lr){6-9} \cmidrule(lr){10-18}
 & 25 & 15 & 10 & 10 & 3           & 1.30 & 246.63        & \num{585368335} & 2  & 1.39 & 0 & & 8  & 2.45   & \textbf{10} & 3.35   & \num{7519692}   \\
 & 25 & 20 & 10 & 10 & 6           & 1.28 & 398.36        & \num{795107155} & 1  & 1.22 & 0 & & 9  & 3.65   & \textbf{10} & 4.51   & \num{10153627}  \\
 & 25 & 20 & 15 & 10 & 0           &      &               &                 & 0  & 1.88 & 0 & & 10 & 202.15 & \textbf{10} & 204.04 & \num{555730619} \\
 & 25 & 25 & 10 & 10 & 7           & 1.53 & 292.63        & \num{532301575} & 3  & 1.51 & 0 & & 7  & 28.39  & \textbf{10} & 21.39  & \num{51459402}  \\
 & 25 & 25 & 15 & 10 & \textbf{10} & 2.82 & \textbf{2.82} & \textbf{0}      & 10 & 3.10 & 0 & & 0  &        & \textbf{10} & 3.10   & \textbf{0}      \\
 & 25 & 25 & 20 & 30 & 21          & 3.59 & 3.59          & \num{0}         & 19 & 1.52 & 0 & & 11 & 35.51  & \textbf{30} & 14.54  & \num{31376692}  \\
\cmidrule(lr){2-5} \cmidrule(lr){6-9} \cmidrule(lr){10-18}
 &   \multicolumn{2}{l}{\textbf{Tot/Avg}} & & 80 & 47 & 2.68 & 112.38 & \num{218145935} & 35 & 1.71 & 0 & & 45 & 59.18 & \textbf{80} & 35.00 & \num{89874177} \\
\midrule
\multicolumn{4}{l}{\textbf{Tot/Avg}} & 600 & 508 & 0.56 & 81.01 & \num{142785782} & 180 & 0.38 & 232 & 3.75 & 188 & 15.91 & \textbf{600} & 6.81 & \num{25769536} \\
\bottomrule
\end{tabular}
\end{adjustbox}
\end{table}
Table~\ref{tab:laporteCD} details the comparison. Each line refers to a subset of instances sharing the same values of $n$, $m$, and $c$, whose cardinality is reported in column ``\#Inst''. For each algorithm, the table also reports the average number of nodes explored (columns ``Nodes''), together with the contribution of each single component. For LSS04, column ``$\mathrm{Sec}_{\mathrm{HGS}}$'' gives the average time spent by HGS to compute the initial upper bound. For C-B\&B, the two columns under ``Prep.''\ give the number of instances already closed at the end of the preprocessing phase and the average duration of that phase, whereas those under ``PI-B\&B'' and ``PE-B\&B'' give the number of instances closed by the corresponding enumeration scheme and its average running time on them.

On group $C$, C-B\&B closes all 340 instances in $0.07$ seconds and \num{113706} nodes on average, against the $46.68$ seconds and \num{62350611} nodes of LSS04. The reduction of the search tree is observed on every subset, with factors ranging from 35 to more than 1400. On instances of this size, C-B\&B as a whole is even faster than HGS alone: it takes $0.07$ seconds overall, whereas LSS04 spends $0.17$ seconds just to compute its initial upper bound with HGS. 
On the instances of group $D$ with 20 jobs, C-B\&B closes all 180 instances in $7.02$ seconds on average, whereas LSS04 solves 122 of them, needing $164.29$ seconds on average on those it closes. Observe that the subsets on which LSS04 performs worst are precisely those in which C-B\&B most frequently invokes PI-B\&B, whose enumeration scheme is complementary to the one adopted by LSS04: this confirms that the solver-selection strategy of Section~\ref{sec:ExactAlgorithm} correctly exploits such complementarity. 
The instances with 25 jobs are more homogeneous: their optimal values lie close to the trivial lower bound $m-c$, so that PE-B\&B is always selected for the final exact enumeration. This allows a clean comparison between the components of PE-B\&B introduced in Section~\ref{sec:BBL} and those of LSS04, since the two algorithms share the same branching scheme. Owing to these components, PE-B\&B explores far fewer nodes and closes all 80 instances, whereas LSS04 proves optimality for only 47 of them. Moreover, LSS04 generally solves only the easiest instances of each subset: on the two subsets with $m=25$ and $c\in\{15,20\}$, for example, it closes only the instances for which HGS already attains $m-c$, so that optimality follows at the root node.

\section{Conclusions}\label{sec:Conclusions}
In this paper we proposed C-B\&B, an innovative algorithm for the exact solution of the job sequencing and tool switching problem. The proposed algorithm is purely combinatorial, in that it does not resort to any MILP formulation, and is based on the combination of two complementary B\&B schemes, namely PI-B\&B and PE-B\&B. The first one relies on an insertion-based branching scheme designed to reduce the size of the implicit enumeration tree, which we enriched with a collection of new bounding functions that substantially limit the number of explored nodes. The second one builds on the branching scheme introduced by \citet{LSS2004}, which we strengthened with a new bounding function and two novel dominance rules. We also designed PI-B\&B-H, a heuristic variant of PI-B\&B capable of rapidly generating high-quality solutions for the SSP. We finally combined these algorithms into a single procedure, which starts with a preprocessing phase that incorporates PI-B\&B-H to compute an initial incumbent solution, and then selects which of the two exact schemes to use for the final enumeration, according to the characteristics of the instance.

Extensive computational experiments show that C-B\&B outperforms the exact approaches available in the literature by a large margin, closing, for the first time, all the benchmark instances proposed so far with up to 25 jobs, with a substantial reduction in computing time. Moreover, on the instances with at most 15 jobs, it is even faster than the most effective heuristic available for the problem, namely the HGS of \citet{Mecler2021}. To the best of our knowledge, C-B\&B is also the first algorithm able to tackle the 30-job instances from the literature, solving four of them to proven optimality within a time limit of only 1200 seconds.


\section*{Acknowledgment}
We acknowledge financial support from the Natural Sciences and Engineering Research Council of Canada (NSERC) under Grant 2026-05294, the University of Modena and Reggio Emilia under Grant FAR-DISMI-25, and the Emilia-Romagna regional funding program FSE+ 2021--2027 under Council Resolution No.~693/2023.

%
%
\bibliographystyle{plainnat}
\bibliography{bib}

@ARTICLE{LSS2004,
author={Laporte, G. and Salazar-González, J. J. and Semet, F.},
title={Exact algorithms for the job sequencing and tool switching problem},
journal = {IIE Transactions},
year={2004},
volume={36},
pages={37-45},
}

@ARTICLE{C2019,
author={Calmels, D.},
title={The job sequencing and tool switching problem: state-of-the-art literature review, classification, and trends},
journal={International Journal of Production Research},
year={2019},
volume={57},
pages={5005-5025},
}

@ARTICLE{CKOS1994,
author={Crama, Y. and Kolen, A.W.J. and Oerlemans, A.G. and Spieksma, F.C.R.},
title={Minimizing the number of tool switches on a flexible machine},
journal={International Journal of Flexible Manufacturing Systems},
year={1994},
volume={6},
pages={33-54},
}

@ARTICLE{TD1988,
author={Tang, C. S. and Denardo, E. V.},
title={Models arising from a flexible manufacturing machine, part {I}: minimization of the number of tool switches},
journal={Operations Research},
year={1988},
volume={36},
pages={767-777},
}

@ARTICLE{ILL2022,
author={Iori, M.
and Locatelli, A.
and Locatelli, M.},
title={A {GRASP} for a real-world scheduling problem with unrelated parallel print machines and sequence-dependent setup times},
journal={International Journal of Production Research},
year={2023},
pages = {7367-7385},
volume = {61}
}

@InProceedings{ILLS2022,
author="Iori, M.
and Locatelli, A.
and Locatelli, M.
and Salazar-Gonz\'alez, J. J.",
title="Tool Switching Problems in the Context of Overlay Printing with Multiple Colours",
booktitle="Proceedings of 7th International Symposium on Combinatorial Optimization (ISCO 2022)",
volume="13526",
year="2022",
pages="260--271",
}

@ARTICLE{L2023,
	author = {Locatelli, A.},
	title = {Optimization methods for knapsack and tool switching problems},
	year = {2023},
	journal = {4OR},
	volume = {21},
	  @COMMENTnumber = {4},
	pages = {715 – 716},
}

@ARTICLE{Ghiani2010,
	author = {Ghiani, G. and Grieco, A. and Guerriero, E.},
	title = {Solving the job sequencing and tool switching problem as a nonlinear least cost {Hamiltonian} cycle problem},
	year = {2010},
	journal = {Networks},
	volume = {55},
	@COMMENTnumber = {4},
	pages = {379 – 385},
}

@ARTICLE{Kruskal1956,
	author = {Kruskal, J. B.},
	title = {On the shortest spanning subtree of a graph and the traveling salesman problem},
	year = {1956},
	journal = {Proceedings of the American Mathematical Society},
	volume = {7},
	@COMMENTnumber = {1},
	pages = {48 – 50},
}

@ARTICLE{Almeida2025,
	author = {Almeida, A. L. Barroso and de Castro Lima, J. and Carvalho, M. A. M.},
    title = {On serial and parallel evaluation functions for Job Sequencing and Tool Switching problems},
    journal = {Computers \& Operations Research},
    volume = {177},
    pages = {106969},
    year = {2025},
}

@ARTICLE{Akhundov2024,
	author = {Akhundov, N. and Ostrowski, J.},
	title = {Exploiting symmetry for the job sequencing and tool switching problem},
	year = {2024},
	journal = {European Journal of Operational Research},
	volume = {316},
	@COMMENTnumber = {3},
	pages = {976 – 987},
}

@ARTICLE{Iori2024,
	author = {Iori, M. and Locatelli, A. and Locatelli, M. and Salazar-González, J. J.},
	title = {Tool switching problems with tool order constraints},
	year = {2024},
	journal = {Discrete Applied Mathematics},
	volume = {347},
	pages = {249 – 262},
}

@ARTICLE{daSilva2021,
	author = {da Silva, T. T. and Chaves, A. A. and Yanasse, H. H.},
	title = {A new multicommodity flow model for the job sequencing and tool switching problem},
	year = {2021},
	journal = {International Journal of Production Research},
	volume = {59},
	@COMMENTnumber = {12},
	pages = {3617 – 3632},
}

@ARTICLE{Belady1966,
  author={Belady, L. A.},
  journal={IBM Systems Journal}, 
  title={A study of replacement algorithms for a virtual-storage computer}, 
  year={1966},
  volume={5},
  @COMMENTnumber={2},
  pages={78-101},}

@article{mutze2014,
  title={Scheduling with few changes},
  author={M{\"u}tze, T.},
  journal={European Journal of Operational Research},
  volume={236},
  @COMMENTnumber={1},
  pages={37--50},
  year={2014},
}

@ARTICLE{Privault1995,
	author = {Privault, C. and Finke, G.},
	title = {Modelling a tool switching problem on a single {NC}-machine},
	year = {1995},
	journal = {Journal of Intelligent Manufacturing},
	volume = {6},
	@COMMENTnumber = {2},
	pages = {87 – 94},}

@inproceedings{legrand2025,
  title={A Dynamic Programming Approach for the Job Sequencing and Tool Switching Problem},
  author={Legrand, E. and Copp{\'e}, V. and Catanzaro, D. and Schaus, P.},
  booktitle={International Conference on the Integration of Constraint Programming, Artificial Intelligence, and Operations Research},
  pages={70--85},
  year={2025},
  organization={Springer}
}

@ARTICLE{Mecler2021,
  title={A simple and effective hybrid genetic search for the job sequencing and tool switching problem},
  author={Mecler, J. and Subramanian, A. and Vidal, T.},
  journal={Computers \& Operations Research},
  volume={127},
  pages={105153},
  year={2021}
  }

@ARTICLE{Qiu2026,
	author = {Qiu, Y. and Cherniavskii, M. and Goldengorin, B. and Pardalos, P. M.},
	title = {A Computational Study of the Tool Replacement Problem},
	year = {2026},
	journal = {INFORMS Journal on Computing},
	volume = {38},
	@COMMENTnumber = {1},
	pages = {86 – 101},
}

@ARTICLE{Hadj2026,
  author = {Hadj Salem, K. and Kramer, A. and Robbes, A.},
  journal = {European Journal of Operational Research},
  title = {Job sequencing and tool switching problem with non-identical parallel machines: Mathematical formulations and modeling improvements},
  year = {2026},
  volume = {330},
  @COMMENTnumber = {2},
  pages = {416--426},
}

@ARTICLE{Soares2024,
  author = {Soares, L. C. R. and Carvalho, M. A. M.},
  journal = {Computers \& Operations Research},
  title = {Biased random-key genetic algorithm for the job sequencing and tool switching problem with non-identical parallel machines},
  year = {2024},
  volume = {163},
  pages = {106509},
}

@ARTICLE{Cura2023,
  author = {Cura, T.},
  journal = {Expert Systems with Applications},
  title = {Hybridizing local searching with genetic algorithms for the job sequencing and tool switching problem with non-identical parallel machines},
  year = {2023},
  volume = {223},
  pages = {119908},
}

@ARTICLE{Calmels2022,
  author = {Calmels, D.},
  journal = {European Journal of Operational Research},
  title = {An iterated local search procedure for the job sequencing and tool switching problem with non-identical parallel machines},
  year = {2022},
  volume = {297},
  @COMMENTnumber = {1},
  pages = {66--85},
}

@ARTICLE{RIFAI2022,
  author = {Rifai, A. P. and Windras Mara, S. T. and Norcahyo, R.},
  journal = {Computers \& Industrial Engineering},
  title = {A two-stage heuristic for the sequence-dependent job sequencing and tool switching problem},
  year = {2022},
  volume = {163},
  pages = {107813},
}

@article{Yanasse2009,
  title     = {Um algoritmo enumerativo baseado em ordenamento parcial para resolu{\c{c}}{\~a}o do problema de minimiza{\c{c}}{\~a}o de trocas de ferramentas},
  author    = {Yanasse, H. H. and Rodrigues, R. de C. M. and Senne, E. L. F.},
  journal   = {Gest{\~a}o \& Produ{\c{c}}{\~a}o},
  volume    = {16},
@COMMENTnumber    = {3},
  pages     = {370--381},
  year      = {2009},
}

@ARTICLE{Catanzaro2015766,
	author = {Catanzaro, D. and Gouveia, L. and Labbé, M.},
	title = {Improved integer linear programming formulations for the job Sequencing and tool Switching Problem},
	year = {2015},
	journal = {European Journal of Operational Research},
	volume = {244},
	@COMMENTnumber = {3},
	pages = {766 – 777},
}

@article{tarjan1975,
	author = {Tarjan, R. E.},
	title = {Efficiency of a Good But Not Linear Set Union Algorithm},
	year = {1975},
	journal = {Journal of the ACM (JACM)},
	volume = {22},
	@COMMENTnumber = {2},
	pages = {215 – 225},
}

\end{document}